\documentclass[a4paper,leqno, oneside]{amsart}
\usepackage[letterpaper,top=2cm,bottom=2cm,left=3cm,right=3cm,marginparwidth=1.75cm]{geometry}
\usepackage[OT1,T2A]{fontenc}

\usepackage[english]{babel}
\usepackage{amsmath,amssymb,amsthm,amscd}
\usepackage{latexsym}
\usepackage{mathrsfs}
\usepackage{array}
\usepackage[dvipdfmx]{graphicx}
\usepackage{color}
\usepackage{mathtools}
\usepackage{tikz-cd} 
\DeclareMathOperator{\Spec}{Spec}
\DeclareMathOperator{\Hom}{Hom}
\DeclareMathOperator{\End}{End}

\DeclareMathOperator{\GL}{GL}
\DeclareMathOperator{\Gr}{Gr} 

\DeclareMathOperator{\Ker}{Ker}
\DeclareMathOperator{\Img}{Im}
\DeclareMathOperator{\Aut}{Aut}
\DeclareMathOperator{\Inn}{Inn}
\DeclareMathOperator{\Out}{Out}

\DeclareMathOperator{\Gal}{Gal}
\DeclareMathOperator{\Ann}{Ann}

\newcommand{\mbC}{\mathbb{C}}

\newcommand{\mbQ}{\mathbb{Q}}

\newcommand{\mfg}{\mathfrak{g}}

\newcommand{\p}{\mathfrak{p}}

\newcommand{\et}{\mathrm{\acute{e}t}} 
\newcommand{\ab}{\mathrm{ab}} 

\numberwithin{equation}{subsection}

\definecolor{e-mail}{rgb}{0,.40,.80}
\definecolor{reference}{rgb}{.20,.60,.22}
\definecolor{mrnumber}{rgb}{.80,.40,0}
\definecolor{citation}{rgb}{0,.40,.80}
\usepackage[breaklinks=true,colorlinks=true,linkcolor=reference, citecolor=citation, urlcolor=e-mail]{hyperref}

\newtheorem{mainthm}{Theorem}

\theoremstyle{plain}
\newtheorem{theorem}{Theorem}[section]
\newtheorem{lemma}[theorem]{Lemma}
\newtheorem{proposition}[theorem]{Proposition}
\newtheorem{corollary}[theorem]{Corollary}
\newtheorem{conjecture}[theorem]{Conjecture}

\theoremstyle{definition} 
\newtheorem{definition}[theorem]{Definition}
\newtheorem{example}[theorem]{Example}
\newtheorem{remark}{Remark}[theorem]

\usepackage{url}

\title{On the kernels of the pro-$p$ outer Galois representations associated to once-punctured CM elliptic curves}
\author[S.~Ishii]{Shun Ishii}
\address{Department of Mathematics, Keio University, 3-14-1 Hiyoshi, Kouhoku-ku, Yokohama 223-8522, Japan.}
\email{ishii.shun@keio.jp}

\date{\today}

\begin{document}
\begin{abstract}
We prove that a certain field naturally arising from the kernel of the pro-$p$ outer Galois representation of a once-punctured elliptic curve with complex multiplication over an imaginary quadratic field $K$ coincides with the maximal pro-$p$ Galois extension of the mod-$p$ ray class field $K(p)$ of $K$ unramified outside $p$ under suitable assumptions. This result gives a geometric way to construct a large pro-$p$ extension of $K(p)$ with restricted ramification, and may be regarded as an analogue of Sharifi's result in the case of the projective line minus three points.
\end{abstract}

\maketitle

\setcounter{tocdepth}{1}
\tableofcontents

\section{Introduction}\label{1}

In this paper, we study the pro-$p$ outer Galois representations associated to once-punctured elliptic curves with complex multiplication, especially their kernels. Let us briefly recall the definition of the pro-$p$ outer Galois representations: Suppose that $X$ is a geometrically connected algebraic variety defined over a number field $F$. Throughout the paper, we work with a fixed algebraic closure $\bar{\mathbb{Q}}$ of $\mathbb{Q}$, and every number field such as $F$ is considered to be a subfield of $\bar{\mathbb{Q}}$.

We denote the \'etale fundamental group of $X$ with respect to a (possibly tangential) basepoint $\bar{x}$ by $\pi_{1}(X, \bar{x})$. In the following, we write $\bar{X} \coloneqq X \times_{\Spec(F)} \Spec(\bar{\mathbb{Q}})$. There is an exact sequence determined by the structure morphism $X \to \Spec(F)$, called the \emph{\'etale homotopy exact sequence}:
\[
    1 \to \pi_{1}(\bar{X}, \bar{x}) \to \pi_{1}(X, \bar{x}) \to G_{F} \coloneqq \Gal(\bar{\mathbb{Q}}/F) \to 1.  
\]
This exact sequence, together with the conjugation action of $\pi_{1}(X, \bar{x})$ on $\pi_{1}(\bar{X}, \bar{x})$, determines the following \emph{outer Galois representation} 
\[
\rho_{X} \colon G_{F} \to \Out(\pi_{1}(\bar{X}, \bar{x})) \coloneqq \Aut(\pi_{1}(\bar{X}, \bar{x}))/\Inn(\pi_{1}(\bar{X}, \bar{x})),
\] which does not depend on the choice of basepoints. For a prime $p$, the outer representation $\rho_{X}$ induces a homomorphism
\[
    \rho_{X, p} \colon G_{F} \to \Out(\pi_{1}(\bar{X}, \bar{x})^{(p)}),
\]
 which we call the \emph{pro-$p$ outer Galois representation} associated to $X$. Roughly speaking, since $\pi_{1}(\bar{X}, \bar{x})^{(p)}$ is the fundamental group of the Galois category of geometrically $p$-coverings of $\bar{X}$, the outer representation $\rho_{X,p}$ expresses how $G_{F}$ acts on such a category or such coverings via pullbacks.

If $X$ is a hyperbolic curve, these outer representations are studied in the context of anabelian geometry. Among other things, it is known that $\rho_{X}$ is injective if $X$ is a hyperbolic curve, cf. Matsumoto \cite{Ma} when $X$ is affine and Hoshi-Mochizuki \cite{HoM} when $X$ is proper. Regarding the pro-$p$ outer Galois representation, it is far from being injective since the group $\Out(\pi_{1}(\bar{X}, \bar{x})^{(p)})$ contains an open pro-$p$ subgroup. In particular, the fixed field of the kernel of $\rho_{X,p}$ is an almost pro-$p$ extension over $F$, and it seems to be interesting to study arithmetic properties of this fixed field.

 In the case of the thrice-punctured projective line $\mathbb{P}^{1}_{\mathbb{Q}} \setminus \{ 0,1,\infty\}$, Anderson and Ihara  \cite[Theorem 2 (I\hspace{-.1em}V)]{AI} proved that the fixed field of the kernel of the pro-$p$ outer Galois representation is a nonabelian infinite pro-$p$ extension over the field $\mathbb{Q}(\mu_{p^{\infty}})$ unramified outside $p$, and is generated by certain $p$-units, which they call \emph{higher circular $p$-units}, constructed from cusps of geometrically $p$-coverings of $\mathbb{P}^{1}_{\bar{\mathbb{Q}}} \setminus \{ 0,1,\infty\}$ of genus zero. Moreover, they posed the following question \cite[page 272, (a)]{AI}:

 \begin{center}
    \emph{Is $\bar{\mathbb{Q}}^{\Ker(\rho_{\mathbb{P}^{1}_{\mathbb{Q}} \setminus \{ 0,1,\infty\}, p})}$ the maximal pro-$p$ extension of $\mathbb{Q}(\mu_{p})$ unramified outside $p$?}
 \end{center}

Let $\Omega^{\mathrm{cyc}}$ denote the maximal pro-$p$ extension of the $p$-th cyclotomic field $\mathbb{Q}(\mu_{p})$ unramified outside $p$. Regarding Anderson-Ihara's question, Sharifi \cite{Sh} has established the following affirmative result for odd regular primes, assuming a certain conjecture which is nowadays referred to as \emph{the Deligne-Ihara conjecture} (it was called \emph{Deligne's conjecture} in \cite{Sh}):

\begin{theorem}[Sharifi {\cite[Theorem 1.1]{Sh}}]
Let $p$ be an odd regular prime and suppose the Deligne-Ihara conjecture holds for $p$. Then the fixed field of the pro-$p$ outer Galois representation associated to the projective line minus three points coincides with $\Omega^{\mathrm{cyc}}$.
\end{theorem}

The Deligne-Ihara conjecture concerns with the structure of a graded Lie algebra over $\mathbb{Q}_{p}$ associated to a certain descending central filtration on $G_{\mathbb{Q}}$ induced by the pro-$p$ outer Galois representation, and states that the graded Lie algebra is freely generated by Soul\'e elements $\{ \sigma_{m} \}_{m \geq 3, \mathrm{odd}}$ in each odd degree $\geq 3$. We do not go into further details here (see Sharifi \cite{Sh} and Ihara \cite{Ih6} for detailed discussions), but we will formulate an analogue of the Deligne-Ihara conjecture for once-punctured elliptic curves with complex multiplication in the next section (Conjecture \ref{cnj:DI11}).

The Deligne-Ihara conjecture is now a theorem: The generation portion of the conjecture is proved by Hain-Matsumoto \cite{HM2}, and later Brown \cite{Br} proved the conjecture in full generality. Anderson-Ihara's question is hence affirmative whenever $p$ is odd and regular. 

\begin{remark}\label{rmk:previous}
Let $p$ be an odd prime. The Galois group $\Gal(\Omega^{\mathrm{cyc}}/\mathbb{Q}(\mu_{p}))$ is generated by $(p+1)/2$ elements if and only if $p$ is regular by \cite[(10.7.13) Theorem]{NSW}. This observation plays a central role in the proof of \cite[Theorem 1.1]{Sh} Moreover, if $p$ is odd and regular, then $\Gal(\Omega^{\mathrm{cyc}}/\mathbb{Q}(\mu_{p}))$ is even a free pro-$p$ group of rank $(p+1)/2$ and the $p$-ramified Iwasawa module $\Gal(\Omega^{\mathrm{cyc}}/\mathbb{Q}(\mu_{p^{\infty}}))^{\ab}$ is generated by $\frac{p-1}{2}$ elements by \emph{loc. cit.}. One can even show that the $p$-ramified Iwasawa module is free (of rank $\frac{p-1}{2}$) by using \cite[Corollaire 2.7]{NQD}. 
\end{remark}

The aim of the present paper is to establish an analogue of Sharifi's result in the case of once-punctured elliptic curves with complex multiplication defined over imaginary quadratic fields of class number one. 

In the following, let $K$ be an imaginary quadratic field of class number one, and $(E,O)$ an elliptic curve over $K$ with the origin $O \in E(K)$ having complex multiplication by the ring of integers $O_{K}$. We assume that the complex multiplication by $O_{K}$ is defined over $K$, i.e. $O_{K} \hookrightarrow \End_{K}(E)$. For a prime $p$ and a nonnegative integer $n$, we write $K(p^{n})$ for the ray class field of $K$ of conductor $p^{n}$, and set $K(p^{\infty}) \coloneqq \cup_{n \geq 0} K(p^{n})$. Then $K(p^{\infty})/K(p)$ is a $\mathbb{Z}_{p}^{2}$-extension.

Our main result characterizes the kernel of the pro-$p$ outer Galois representation associated to the complement $X \coloneqq E-O$ of the origin of the elliptic curve, which we call the \emph{once-punctured elliptic curve} associated to $E$.

\begin{mainthm}[=Theorem \ref{thm:main}]\label{thm:A}
Let $p \geq 5$ be a prime which splits in $K$. Assume that
\begin{enumerate}
\item[(1)] the class number of $K(p)$ is not divisible by $p$,
\item[(2)] there are only two primes of $K(p^{2})$ lying above $p$, and
\item[(3)] an analogue of the Deligne-Ihara conjecture (Conjecture \ref{cnj:DI11}) holds. 
\end{enumerate}
Then the fixed field of the kernel of the pro-$p$ outer Galois representation of $X=E-O$ coincides with the compositum of $K(E[p])$ and the maximal pro-$p$ extension of $K(p)$ unramified outside $p$.
\end{mainthm}

In the following, we write $\Omega$ for the maximal pro-$p$ extension of $K(p)$ unramified outside $p$. To the author's knowledge, though Theorem \ref{thm:A} critically depends on our working Conjecture \ref{cnj:DI11}, the pair $(X,p)$ as above gives a first example where the fixed field is explicitly determined independently of the case of $(\mathbb{P}^{1}_{\mathbb{Q}} \setminus \{ 0,1,\infty\}, p)$. Note that one can construct many hyperbolic curves of higher genera whose kernels of pro-$p$ outer Galois representations are the same as that of the projective line minus three points by using ``Oda's prediction'' \cite[Theorem 3.6]{Ta12} (or \cite[Theorem C (i)]{HoM}) and \cite[Lemma 28]{Ho12}. The Fermat curve of degree $p$ provides one such example \cite[Example 3.4 (2)]{Ho}.

\begin{remark}\label{rmk:main}
  Let $p \geq 5$ be a prime which splits into distinct primes in $K$. 
\begin{enumerate}
    \item The first and second assumptions of Theorem \ref{thm:A}, together with \cite[(10.7.13) Theorem]{NSW}, imply that the pro-$p$ group $\Gal(\Omega/K(p))$ is generated by $[K(p):K]+2$ elements satisfying a single relation.
    \item The first and second assumptions of Theorem \ref{thm:A} also allow us to determine the $\Lambda$-module structure of the $p$-ramified Iwasawa module $\Gal(\Omega/K(p^{\infty}))^{\ab}$ completely, where \[
    \Lambda \coloneqq \mathbb{Z}_{p}[[\Gal(K(p^{\infty})/K(p))]]
    \] is non-canonically isomorphic to the two-variable power series ring over $\mathbb{Z}_{p}$. More precisely, under these two assumptions, we have 
    \[
    \Gal(\Omega/K(p^{\infty}))^{\ab} \cong \Lambda^{[K(p):K]-1} \oplus \Ann_{\Lambda}(\mathbb{Z}_{p}(1)),
    \] where $\Ann_{\Lambda}(\mathbb{Z}_{p}(1))$ denotes the annihilator ideal of $\mathbb{Z}_{p}(1)$, cf. Proposition \ref{prp:Iwasawa}.
    \item We do not know whether there are infinitely many primes satisfying the first assumption of Theorem \ref{thm:A}, simply because the assumption implies that $p$ is regular. We also do not know whether the second assumption is satisfied by infinitely many primes. There are 39175 primes which are less than $10^{6}$ and split in $\mathbb{Q}(\sqrt{-1})$, and among them, 13705 primes satisfy the second assumption.  
    \item It is desirable to generalize Theorem \ref{thm:A} to more general imaginary quadratic fields. One might try to replace $X$ with a once-punctured CM elliptic curve over the Hilbert class field of the concerned imaginary quadratic field. 
\end{enumerate}
\end{remark}

\begin{example}\label{ex:main}
   We have checked whether a small prime $p$ satisfies the first two assumptions of Theorem \ref{thm:A} by using SageMath \cite{Sage}. For example, when $p=5$, both assumptions are satisfied if and only if $K=\mathbb{Q}(\sqrt{-1})$. Similarly, when $p=7$, both conditions are satisfied only for $K=\mathbb{Q}(\sqrt{-3})$ or $\mathbb{Q}(\sqrt{-19})$. When $p=13$ or $17$, we have also checked that both assumptions are satisfied for $K=\mathbb{Q}(\sqrt{-1})$ under the Generalized Riemann Hypothesis (GRH).  
\end{example}

Our strategy to prove Theorem \ref{thm:A} is to generalize Sharifi's technique developed in \cite{Sh}. That is, we construct certain elements $\{ \sigma_{\boldsymbol{m}} \}_{\boldsymbol{m} \in I}$ that satisfy the assumption of Theorem \ref{thm:A} in such a group-theoretic way that they generate $\Gal(\Omega/K(p^{\infty}))$. Then Conjecture \ref{cnj:DI11} allows us to conclude that the surjective homomorphism $\Gal(\Omega/K(p^{\infty})) \to \Gal(\Omega^{\ast}/K(p^{\infty}))$ is an isomorphism.

However, to construct such elements, we need to introduce a new ingredient. Namely, we define certain two-variable filtrations on the pro-$p$ geometric fundamental group $\Pi_{1,1}$ of $X$ and on the absolute Galois group $G_{K}$. We introduce them and establish their fundamental properties in Section \ref{3}. 

As was mentioned in Remark \ref{rmk:main} (1), in our situation, generators of the Galois group $\Gal(\Omega/K(p))$ satisfy a nontrivial relation. This is one of the significant differences from the previous result on the thrice-punctured projective line (cf. Remark \ref{rmk:previous}), making it difficult to extend Sharifi's approach to the case of once-punctured elliptic curves. We overcome this difficulty by choosing elements $\{ \sigma_{\boldsymbol{m}} \}_{\boldsymbol{m} \in I}$ carefully in Section \ref{4}. Once we obtain $\{ \sigma_{\boldsymbol{m}} \}_{\boldsymbol{m} \in I}$, the rest of the proof follows in the same way as the previous result.

This paper proceeds as follows: In Section \ref{2}, we explain previous studies on the pro-$p$ outer Galois representations associated to once-punctured elliptic curves. Among other result, we review the construction of a certain power series obtained from the Galois action on the metabelian fundamental group and state Nakamura's explicit formula for that power series. We also propose an analogue of the Deligne-Ihara conjecture and state the main result in a precise manner in that section. In Section \ref{3}, we introduce a two-variable version of the descending central series on various profinite groups and establish their fundamental properties. Section \ref{4} is devoted to the proof of Theorem \ref{thm:A}, and in the appendix, we obtain a certain sufficient condition for the so-called \emph{pure locality} for $K(\mu_{p})$, which is needed to establish the finiteness of certain Galois cohomology groups.

\medskip

{\bf Acknowledgement.} The present paper is based on a part of the author's doctoral thesis submitted to Kyoto University. The author sincerely thanks his advisor \emph{Akio Tamagawa} for helpful comments and warm encouragement, and \emph{Benjamin Collas} for carefully reading drafts of the thesis. He also sincerely appreciates the anonymous referees for their constructive comments, which have greatly contributed to improving the content of the initial manuscript of the paper. The author was supported by JSPS KAKENHI Grant Number JP21J1191 when preparing his thesis, and by JP23KJ1882 throughout the writing of the the present paper.

\section*{Notational conventions}

\subsection*{Indexes}

Let $\boldsymbol{m}=(m_{1}, m_{2})$ be a pair of integers and $w$ a positive integer. We write
\begin{itemize}
    \item $\boldsymbol{1}$ for the pair $(1,1)$, 
    \item $\lvert \boldsymbol{m} \rvert \coloneqq m_{1}+m_{2}$,
    \item $\boldsymbol{m} \equiv 0 \bmod w$ if both $m_{1}$ and $m_{2}$ are divisible by $w$.
\end{itemize}
Suppose  $\boldsymbol{n}=(n_{1}, n_{2})$ is another pair of integers. We write 
\begin{itemize}
    \item $\boldsymbol{m} \geq \boldsymbol{n}$ if $m_{i} \geq n_{i}$ for $i = 1, 2 $,
    \item $\boldsymbol{m} >  \boldsymbol{n}$ if $\boldsymbol{m} \geq \boldsymbol{n}$ and $\lvert \boldsymbol{m} \rvert > \lvert \boldsymbol{n} \rvert$.
\end{itemize}

\subsection*{Profinite Groups} 

Let $G$ be a profinite group, and $H, K$ two closed subgroups of $G$. We write
\begin{itemize}
    \item $[H,K]$ for the closure of the commutator subgroup of $H$ and $K$,
    \item $\{ G(m) \}_{m \geq 1}$ for the descending central series defined by
    \[
    G(1) \coloneqq G \quad \text{and} \quad
    G(m) \coloneqq \langle [G(m'), G(m'')] \mid m'+m''=m \rangle  \quad (m \geq 2),
    \]
    \item $G^{\ab}$ for the maximal abelian quotient $G/G(2)$,
    \item  $G^{\rm{mab}}$ for the maximal metabelian quotient $G/[G(2),G(2)]$ of $G$. 
    \item $G^{(p)}$ for the maximal pro-$p$ quotient of $G$ (for a prime $p$),
    \item $\Aut(G)$ for the (continuous) automorphism group of $G$,
    \item $\Inn(G)$ for the inner automorphism group of $G$,
    \item $\Out(G)$ for the outer automorphism group $\Aut(G)/\Inn(G)$ of $G$.
\end{itemize}

Let $S$ be a subset of $G$. We write
\begin{itemize}
    \item  $\langle S \rangle$ for the minimal closed subgroup of $G$ that contains $S$, 
    \item $\langle S \rangle_{{\rm normal}}$ for the minimal normal closed subgroup of $G$ that contains $S$.
\end{itemize}

Moreover, we say
\begin{itemize}
    \item $S$ generates $G$ if $\langle S \rangle$ coincides with $G$,
    \item $S$ \emph{strongly} generates $G$ if $S$ generates $G$ and converges to $1$ (cf. Ribes-Zalesskii \cite[2.4]{RZ}),
    \item $S$ \emph{normally} generates $G$ if $\langle S \rangle_{{\rm normal}}$ coincides with $G$.
\end{itemize}

Free pro-$p$ groups of countably infinite rank often appear in the paper. We refer to Ribes-Zalesskii \cite[Lemma 3.3.4]{RZ} for the characterization of free pro-$p$ groups on sets converging to $1$. 

\subsection*{Number Fields}

We work with a fixed algebraic closure $\bar{\mathbb{Q}}$ of $\mathbb{Q}$ and a fixed embedding from $\bar{\mathbb{Q}}$ into $\mathbb{C}$. Every number field is considered to be a subfield of $\bar{\mathbb{Q}}$, and hence of $\mbC$. For an integer $m \geq 1$, we denote the group of $m$-th roots of unity in $\bar{\mathbb{Q}}$ by $\mu_{m}$. Let $F$ be a subfield of $\bar{\mbQ}$ and $v$ a finite place of $F$. We write
\begin{itemize}
    \item $O_{F}$ for the integer ring of $F$,
    \item $F_{v}$ for the $v$-adic completion of $F$,
    \item $G_{F}$ for the absolute Galois group  $\Gal(\bar{\mathbb{Q}}/F)$ of $F$.
\end{itemize}

Let $K$ be an imaginary quadratic field and $\mathfrak{m}$ a nonzero integral ideal of $K$. We write
\begin{itemize}
    \item  $K(\mathfrak{m})$ for the ray class field of $K$ of modulo $\mathfrak{m}$. If $\alpha$ generates $\mathfrak{m}$, then we write it as $K(\alpha)$, 
    \item $K(\mathfrak{m}^{\infty}) \coloneqq \cup_{n \geq 1} K(\mathfrak{m}^{n})$.
\end{itemize}
 
\subsection*{Elliptic Curves with Complex Multiplication}(cf. Silverman \cite[Chapter II]{Si})

Let $K$ be an imaginary quadratic field of class number one, and $E$ an elliptic curve over $K$ having complex multiplication by $O_{K}$. For an ideal $\mathfrak{m}$ of $O_{K}$, the $G_{K}$-action on the $\mathfrak{m}$-torsion subgroup $E[\mathfrak{m}]$ determines an injective homomorphism
\[
\Gal(K(E[\mathfrak{m}])/K) \hookrightarrow \Aut(E[\mathfrak{m}](\bar{\mathbb{Q}})) \cong (O_{K}/\mathfrak{m})^{\times}.
\] Moreover, it induces an isomorphism
\[
\Gal(K(\mathfrak{m})/K) \xrightarrow{\sim} (O_{K}/\mathfrak{m})^{\times}/\Img(O_{K}^{\times}),
\] which does not depend on the choice of $E$. For a prime $p$, we denote the $p$-adic Tate module of $E$ by $T_{p}(E)$.  If $p$ splits into two distinct primes in $K$ as $(p)=\p \bar{\p}$, let $T_{\p}(E)$ (resp. $T_{\bar{\p}}(E)$) denote the inverse limit $\varprojlim_{n} E[\p^{n}](\bar{\mathbb{Q}})$ (resp. $\varprojlim_{n} E[\bar{\p}^{n}](\bar{\mathbb{Q}})$) whose transition maps are taken to be multiplication by $p$. They determine two characters
\[
\chi_{1} \colon G_{K} \to \Aut(T_{\p}(E)) \cong \mathbb{Z}_{p}^{\times} \quad \text{and} \quad
\chi_{2} \colon G_{K} \to \Aut(T_{\bar{\p}}(E)) \cong \mathbb{Z}_{p}^{\times}.
\]  Note that $\chi_{1}\chi_{2}$ coincides with the $p$-adic cyclotomic character $\chi_{\mathrm{cyc}}$. Let $\boldsymbol{m}=(m_{1}, m_{2})$ be a pair of integers. We define the character $\chi^{\boldsymbol{m}}$ by
\[
\chi^{\boldsymbol{m}} \coloneqq \chi_{1}^{m_{1}}\chi_{2}^{m_{2}} \colon G_{K} \to \mathbb{Z}_{p}^{\times}.
\] Note that it factors through $\Gal(K(p^{\infty})/K)$ if $m_{1} \equiv m_{2} \bmod \lvert O_{K}^{\times} \rvert$.

Let $M$ be a $\mathbb{Z}_{p}$-module on which $G_{K}$ acts continuously. we denote the $\chi^{\boldsymbol{m}}$-twist of $M$ by $M(\boldsymbol{m})$. Note that the $\p$-adic (resp. $\bar{\p}$-adic) Tate module $T_{\p}(E)$ is isomorphic to $\mathbb{Z}_{p}(1,0)$ (resp. $\mathbb{Z}_{p}(0,1)$), and $\mathbb{Z}_{p}(m,m)$ is simply the $m$-th Tate twist $\mathbb{Z}_{p}(m)$.

\section{Preliminaries}\label{2}

In this section, we prepare backgrounds which are necessary to state the main result of this paper (Theorem \ref{thm:main}). We explain:

\begin{itemize}
\item a certain basis $\{x_{1}, x_{2} \}$ of the pro-$p$ geometric fundamental group of once-punctured elliptic curve $X$ (Lemma \ref{lmm:basis}),
\item an analogue of the Deligne-Ihara conjecture (Conjecture \ref{cnj:DI11}),
\item an elliptic analogue of the universal power series for Jacobi sums \cite{Ih1},  the explicit formula for that power series \cite{Na2}, conditional nonvanishing of certain Kummer characters associated to the power series \cite{IS1/2} and a reformulation of Conjecture \ref{cnj:DI11} in terms of these characters (Proposition \ref{prp:DI11_Soule}),
\item the main result (Theorem \ref{thm:main}) in a precise manner.
\end{itemize}

In the rest of the present paper, let $K$ be an imaginary quadratic field of class number one, and $p \geq 5$ a prime which splits in $O_{K}$ as $(p)=\p\bar{\p}$. Moreover, let $E$ be an elliptic curve defined over $K$ having complex multiplication by $O_{K}$ with a fixed Weierstrass form of $y^{2}=4x^{3}-g_{2}x-g_{3}$ with $g_{2}, g_{3} \in K$. Let $L$ denote the period lattice corresponding to this form.

We denote the once-punctured elliptic curve associated to $E$ by $X \coloneqq E - O$, where $O$ is the origin of $E$. We denote the pro-$p$ geometric fundamental group $\pi_{1}(\bar{X}, \vec{O})^{(p)}$ with respect to the Weierstrass tangential basepoint at $O$ \cite[(2.7) Case 2]{Na3} by $\Pi_{1,1}$. Note that $\Pi_{1,1}$ is a free pro-$p$ group of rank two, and we have an injective homomorphism $\mathbb{Z}_{p}(1) \subset \Pi_{1,1}$ associated with the tangential basepoint $\vec{O}$. Moreover, the pro-$p$ geometric fundamental group of $E$ is isomorphic to $\Pi_{1,1}^{\rm{ab}}$ through a natural homomorphism induced by the inclusion $X \subset E$, and there is a natural isomorphism $T_{p}(E) \xrightarrow \sim \Pi_{1,0}$. Throughout this paper, we work with the following basis $\{ x_{1}, x_{2} \}$ of $\Pi_{1,1}$, which is useful when considering the outer Galois action on $\Pi_{1,1}$:

\begin{lemma}\label{lmm:basis}
There exists a basis $\{x_{1}, x_{2} \}$ of $\Pi_{1,1}$ satisfying the following conditions:
\begin{enumerate}
\item[(1)] Let $(\omega_{1,n})_{n \geq 1}$ (resp. $(\omega_{2,n})_{n \geq 1}$) denote the image of $x_{1}$ (resp. $x_{2}$) in $\Pi_{1,1}^{\ab} \cong T_{p}(E)= T_{\p}(E) \oplus T_{\bar{\p}}(E)$. Then $(\omega_{1,n})_{n \geq 1}$ (resp. $(\omega_{2,n})_{n \geq 1}$) generates $T_{\p}(E)$ (resp. $T_{\bar{\p}}(E)$).
\item[(2)]  $z \coloneqq [x_{2}, x_{1}]$ generates the inertia subgroup $\mathbb{Z}_{p}(1) \subset \Pi_{1,1}$ determined by $\vec{O}$.
\end{enumerate}
\end{lemma}
\begin{proof}
By a result of Kaneko \cite[Proposition 2]{Ka}, we can choose a basis $\{ x_{1}, x_{2} \}$ so that it satisfies (1) and the subgroup generated by $[x_{1},x_{2}]$ is conjugate to the inertia subgroup determined by $\vec{O}$.  Since the inner automorphism group acts trivially on the maximal abelian quotient $\Pi_{1,1}^{\ab}$, we may replace $\{ x_{1}, x_{2} \}$ with its conjugate so that it also satisfies (2). This concludes the proof.
\end{proof}

\subsection{Structure of graded Lie algebra}\label{2.1}

In this subsection, we introduce a certain filtration on the absolute Galois group $G_{K}$ and formulate a working hypothesis on the structure of the graded Lie algebra associated to this filtration (Conjecture \ref{cnj:DI11}). This is an analogue of the so-called \emph{Deligne-Ihara conjecture} in the case of $\mathbb{P}^{1}_{\mathbb{Q}}-\{0,1, \infty\}$. We then  study some fundamental properties of the fixed field of $\Ker(\rho_{X,p})$ of the pro-$p$ outer Galois representation associated to $X$, and state the main result.

\medskip 

First, we introduce the \emph{pro-$p$ mapping class group of type $(1,1)$}: Denote
\[
\tilde{\Gamma}_{1,1} \coloneqq \{ f \in \Aut(\Pi_{1,1}) \mid  \text{$f$ preserves the conjugacy class of the inertia subgroup $\langle z \rangle$} \}
\] and define the pro-$p$ mapping class group of type $(1,1)$ by the quotient
\[
\Gamma_{1,1} \coloneqq \tilde{\Gamma}_{1,1} / \Inn(\Pi_{1,1}) \subset \Out(\Pi_{1,1}).
\] We define the weight filtration on  $\tilde{\Gamma}_{1,1}$ and $\Gamma_{1,1}$ by 
\[
F^{m}\tilde{\Gamma}_{1,1} \coloneqq \Ker \left( \tilde{\Gamma}_{1,1} \to \Aut(\Pi_{1,1}/\Pi_{1,1}(m+1)) \right)
\quad \text{and} \quad
F^{m}\Gamma_{1,1} \coloneqq F^{m}\tilde{\Gamma}_{1,1} \Inn(\Pi_{1,1})/\Inn(\Pi_{1,1})
\] for each positive integer $m$. They define descending central filtrations, and the intersection $\cap_{m \geq 1} F^{m}\tilde{\Gamma}_{1,1}$ is trivial since $\cap_{m \geq 1} \Pi_{1,1}(m+1)$ is trivial. Moreover, the intersection $\cap_{m \geq 1} F^{m}\Gamma_{1,1}$ is also trivial by a result of Asada \cite[Theorem 2]{As}. 

\medskip

The absolute Galois group $G_{K}$ inherits a descending central filtration $\{ F^{m}G_{K} \}_{m \geq 1}$ by taking a pullback $F^{m}G_{K} \coloneqq \rho_{X,p}^{-1}(F^{m}\Gamma_{1,1})$, and we call it the weight filtration (on $G_{K}$). Note that we have  $F^{1}G_{K}=G_{K(E[p^{\infty}])}$. We form graded quotients and their directed sum as
\[
\mathfrak{g}_{m} \coloneqq F^{m}G_{K}/F^{m+1}G_{K} \text{ for $m \geq 1$ and } \mathfrak{g} \coloneqq \bigoplus_{m \geq 1} \mathfrak{g}_{m}.
\] Since the weight filtration is descending and central, the direct sum $\mathfrak{g}$ is a graded Lie algebra over $\mathbb{Z}_{p}$ whose bracket is induced by commutators. 

We briefly summarize properties of the graded quotients:

\begin{proposition}\label{prp:graded}
    Let $m$ be a positive integer. The following assertions hold:
    \begin{enumerate}
        \item $\mathfrak{g}_{m}$ is trivial whenever $m$ is odd, and it is also trivial when $m=2$.
        \item  $\mathfrak{g}_{m}$ is a free $\mathbb{Z}_{p}$-module of finite rank.
        \item $\mathfrak{g}_{m} \otimes \mathbb{Q}_{p}$ is isomorphic to a direct sum of (a direct sum of) $\mathbb{Q}_{p}(\boldsymbol{m})$'s as $G_{K}$-modules, where $\boldsymbol{m}$ runs over pairs of positive integers satisfying $|\boldsymbol{m}|=m$. In particular, the graded Lie algebra $\mathfrak{g} \otimes \mathbb{Q}_{p}$ has a structure of a bigraded Lie algebra.
    \end{enumerate}
\end{proposition}
\begin{proof}
    (1) The assertion follows from  \cite[(4.2) Proposition]{Na2} and \cite[(4.4)]{Na2}. (2) The assertions follows since $\mathfrak{g}_{m}$ is embedded into $F^{m}\Gamma_{1,1}/F^{m+1}\Gamma_{1,1}$, which is a free $\mathbb{Z}_{p}$-module \cite[Corollary (1.16), $(\mathrm{ii})$]{NT} of finite rank. (3) The assertion follows immediately from the $\GL_{2}(\mathbb{Z}_{p})$-equivariance of the commutative diagram given in \cite[Theorem (1.14)]{NT}.
\end{proof}

\begin{remark}
    The Lie algebra $\mathfrak{g} \otimes \mathbb{Q}_{p}$ only depends on the isomorphism class of $\bar{X}$ not on $X$: There is only one isomorphism class of elliptic curves over $\bar{K}$ having complex multiplication by $O_{K}$ and a base change by a finite extension preserves $\mathfrak{g} \otimes \mathbb{Q}_{p}$. 
\end{remark}

The Deligne-Ihara conjecture states that, a $\mathbb{Q}_{p}$-graded Lie algebra, which is defined similarly to $\mathfrak{g} \otimes \mathbb{Q}_{p}$ using the pro-$p$ outer Galois representation of $\mathbb{P}^{1}_{\mathbb{Q}}-\{ 0,1,\infty \}$ instead that of $X$, is freely generated by one element in each odd degree $m \geq 3$. This conjecture appears in e.g. \cite[p.251, after Theorem (1)]{Ih6} and \cite[the discussion before Theorem 1.1]{Sh}, and there the conjecture was attributed to Deligne's monograph \cite{Del}. We formulate an analogue of this conjecture, as follows. Let us define the set of pairs of integers
\[
I \coloneqq \{ \boldsymbol{m}=(m_{1}, m_{2}) \in \mathbb{Z}_{ \geq 1}^{2} \setminus \{ \boldsymbol{1} \} \mid m_{1} \equiv m_{2} \bmod \lvert O_{K}^{\times} \rvert \}.
\] This set plays a similar role as the set $1+2\mathbb{Z}_{\geq 1}=\{ 3,5,7, \dots \}$ appearing in the case of the projective line minus three points. Our working conjectural description of the Lie algebra $\mfg \otimes \mbQ_{p}$ is  as follows:

\begin{conjecture}\label{cnj:DI11}
            $\mathfrak{g} \otimes \mathbb{Q}_{p}$ is freely generated by one element in each degree in $I$.
\end{conjecture}

\begin{remark}
    (1) One can prove the generation portion of Conjecture \ref{cnj:DI11}, assuming that the second cohomology group $H^{2}_{\et}(O_{K}[1/p], \mathbb{Z}_{p}(\boldsymbol{m}))$ is finite for every $\boldsymbol{m} \in I$, by using (a slightly modified version of) the theory of weighted completions developed by Hain-Matsumoto \cite{HM2}. This provides one evidence for Conjecture \ref{cnj:DI11}, at least for its generation portion. At the writing of the present paper, however, we do not know how to prove the freeness.

    (2) Our proof of Theorem \ref{thm:A} (given in Section \ref{4}) shows that, under the assumption of Theorem \ref{thm:A}, one can even choose a basis of $\mathfrak{g} \otimes \mathbb{Q}_{p}$ so that it also freely generates $\mathfrak{g}$, cf. Corollary \ref{cor:integral} (1). This integral version of Conjecture \ref{cnj:DI11} is expected not to hold in general, and we plan to investigate the integral structure of $\mathfrak{g}$ in future research.
\end{remark}

Although the statement of Conjecture \ref{cnj:DI11} may seem unexpected at first sight, we rephrase it in a different manner in the next subsection by using certain natural characters arising from the $G_{K}$-action on the metabelian fundamental group. The author thinks it gives a more natural statement.

\subsection{Elliptic analogues of Soul\'e characters}\label{2.2}

The aim of this subsection is twofold: First, we study the $G_{K}$-action on the maximal metabelian quotient of $\Pi_{1,1}$ and present some previous results. Then we interpret Conjecture \ref{cnj:DI11} in terms of elliptic analogues of the Soul\'e characters, which is useful in proving our main result. We introduce a certain power series that expresses the Galois action on the maximal metabelian quotient of $\Pi_{1,1}$: This power series is an elliptic analogue of the universal power series for Jacobi sums introduced by Ihara \cite{Ih1}, expressing the Galois action on the metabelian pro-$p$ fundamental group of $\mathbb{P}^{1}_{\mathbb{Q}}-\{ 0,1, \infty \}$. It is extensively studied in the paper of Nakamura \cite{Na2}, and there an application to anabelian geometry of once-punctured elliptic curves is given. We will not go into detail here, and interested readers are encouraged to refer to the papers \cite{Ih1} (in the case of genus zero)  \cite{Na2} (in the case of genus one).

First, we define the subgroup of $\tilde{\Gamma}_{1,1}$ by
\[
\Gamma_{1,1}^{\dag} \coloneqq \{ f \in \Aut(\Pi_{1,1}) \mid \text{$f$ preserves $\langle z \rangle$} \}\footnote{In \cite{Na2}, this subgroup is denoted by $\Gamma_{1,1}^{\ast}$. Since we use the symbol $\ast$ to refer to different kinds of objects in this paper, we use the symbol $\dagger$ instead.}.
\] The weight filtration induces a filtration on $\Gamma^{\dag}_{1,1}$ by taking the intersection
\[
 F^{m}\Gamma^{\dag}_{1,1} \coloneqq \Gamma^{\dag}_{1,1} \cap F^{m}\tilde{\Gamma}_{1,1}. 
\]
Observe that the normalizer subgroup of $\langle z \rangle$ in $\Pi_{1,1}$ coincides with $\langle z \rangle$ itself. It follows that the intersection $F^{m}\Gamma_{1,1}^{\dag} \cap \Inn(\Pi_{1,1})$ is trivial for every $m \geq 3$ and otherwise coincides with $\langle \mathrm{inn}(z) \rangle$. Therefore, the natural projection induces an isomorphism
\[
F^{m}\Gamma_{1,1}^{\dag} \to F^{m}\Gamma_{1,1}
\] for every $m \geq 3$ and, otherwise, the kernel coincides with $\langle \mathrm{inn}(z) \rangle$ \cite[(4.4)]{Na2}. These homomorphisms provide a group-theoretically natural way to lift elements of the pro-$p$ mapping class group.

We also define two subgroups $\Psi^{\dag}_{1}, \Psi^{\dag}$ of the automorphism group of the maximal metabelian quotient of $\Pi_{1,1}$ by
\[
\Psi^{\dag} \coloneqq \{f \in \Aut\left( \Pi_{1,1}^{\mathrm{mab}}\right) \mid \text{$f$ preserves $\langle \bar{z} \rangle$} \}
\quad \text{and} \quad
\Psi^{\dag}_{1} \coloneqq \Ker \left( \Psi^{\dag} \to \Aut(\Pi^{\rm{ab}}) \right),
\]
where $\bar{z} \in \Pi_{1,1}^{\mathrm{mab}}$ denotes the image of $z$. Note that each element $f \in \Psi^{\dag}_{1}$ is uniquely determined by a pair 
$
(f(x_{1})x_{1}^{-1}, f(x_{2})x_{2}^{-1})$. Since $\Pi_{1,1}^{\mathrm{mab}}(2)$ is a free $\mathbb{Z}_{p}[[\Pi_{1,1}^{\ab}]]$-module generated by $\bar{z}$ \cite[Theorem 2]{Ih1}, we can find a unique element $G_{i}(f) \in \mathbb{Z}_{p}[[\Pi_{1,1}^{\ab}]]$ satisfying
\[
f(x_{1})x_{1}^{-1}=G_{1}(f)z 
\quad \text{and} \quad
f(x_{2})x_{2}^{-1}=G_{2}(f)z.
\] In the following, we identify the completed group ring $\mathbb{Z}_{p}[[\Pi_{1,1}^{\rm{ab}}]]$ with the power series ring in two variables $\mathbb{Z}_{p}[[T_{1}, T_{2}]]$ by $T_{i} \coloneqq x_{i}-1$ for $i=1,2$. We then obtain two power series from each element of $\Psi^{\dagger}_{1}$, and these two are related as follows:

\begin{lemma}[Nakamura {\cite[(4.7)]{Na2}}, Tsunogai {\cite[Proposition 1.9]{Ts}}]\label{lmm:BNT}\,
    \begin{enumerate}
        \item For $f \in \Psi^{\dag}_{1}$, two power series $G_{1}(f)$ and $G_{2}(f)$ satisfy the relation
\[
T_{1}G_{2}(f)-G_{1}(f)T_{2}=0.
\] 
\item If we write $H(f) \coloneqq \frac{G_{2}(f)}{T_{2}} = \frac{G_{1}(f)}{T_{1}}$, then we have an isomorphism 
\[
H \colon \Psi_{1}^{\dag} \xrightarrow{\sim} \mathbb{Z}_{p}[[T_{1}, T_{2}]] ; f \mapsto H(f).
\] In particular, the profinite group $\Psi_{1}^{\dag}$ is abelian.
\item The $\GL_{2}(\mathbb{Z}_{p})$-action on $\Psi_{1}^{\dag}$ induced by the  exact sequence
\[
1 \to \Psi_{1}^{\dag} \to \Psi \to \GL_{2}(\mathbb{Z}_{p}) \to 1
\] makes an isomorphism $H \colon \Psi_{1}^{\dag} \xrightarrow{\sim} \mathbb{Z}_{p}[[T_{1}, T_{2}]](\mathrm{det})$ equivariant under the action of $\GL_{2}(\mathbb{Z}_{p})$. Here, $(\mathrm{det})$ denotes the twist by the determinant character.
    \end{enumerate}
 
\end{lemma}

In the following, we shall exploit the section 
\[
s \colon F^{1}G_{K}=G_{K(E[p^{\infty}])} \to F^{3}\Gamma^{\dag}_{1,1}
\] constructed by Nakamura \cite[(4.4)]{Na2} as follows: First, by \cite[(4.4)]{Na2}, we have 
\[
F^{1}\Gamma_{1,1}=F^{2}\Gamma_{1,1}=F^{3}\Gamma_{1,1}.\] Hence the image of $F^{1}G_{K}$ under $\rho_{X,p}$ is contained in $F^{3}\Gamma_{1,1}$. By composing the inverse of the isomorphism $F^{3}\Gamma_{1,1}^{\dag} \xrightarrow{\sim} F^{3}\Gamma_{1,1}$, we obtain the desired homomorphism $s \colon G_{K(E[p^{\infty}])} \to F^{3}\Gamma^{\dag}_{1,1}$. 

\begin{definition}
We define a $\Gal(K(E[p^{\infty}])/K)$-equivariant homomorphism 
\[
\alpha_{1,1} \colon G_{K(E[p^{\infty}])}^{\ab} \to \mathbb{Z}_{p}[[T_{1}, T_{2}]](1)
\] to be a compositum of 
\begin{enumerate}
    \item[(1)] the section $s \colon G_{K(E[p^{\infty}])} \to F^{3}\Gamma_{1,1}^{\dag}$ constructed above,
    \item[(2)] the natural projection $F^{3}\Gamma_{1,1}^{\dag} \to \Psi^{\dag}_{1}$, and
    \item[(3)] the isomorphism $H \colon \Psi^{\dag}_{1} \xrightarrow{\sim} \mathbb{Z}_{p}[[T_{1}, T_{2}]](\mathrm{det})$ constructed in Lemma \ref{lmm:BNT}. 
\end{enumerate}
\end{definition}

In \cite{Na2}, Nakamura obtained the explicit formula of the power series $\alpha_{1,1}$ for general once-punctured elliptic curves. We use the convention $0^{0} \coloneqq 1$ and regard $\mathbb{Z}_{p}[[T_{1}, T_{2}]]$ as a subring of $\mathbb{Q}_{p}[[U_{1}, U_{2}]]$ where $U_{i} \coloneqq \log(1+T_{i})$ in the following formula.

\begin{theorem}[Nakamura {\cite[Theorem (A) and (3.11.5)]{Na2}}]\label{thm:Na} We have
\[
\alpha_{1,1}(\sigma)=\sum_{m \geq 2 \colon \text{even}}^{\infty}\frac{1}{1-p^{m}}\sum_{\substack{\boldsymbol{m}=(m_{1}, m_{2}) \geq (0,0) \\ |\boldsymbol{m}|=m}}\kappa_{\boldsymbol{m}+\boldsymbol{1}}(\sigma)\frac{U_{1}^{m_{1}}U_{2}^{m_{2}}}{m_{1}! m_{2}!}
\] for every $\sigma \in F^{1}G_{K}$, where $\kappa_{\boldsymbol{m}} \colon F^{1}G_{K} \to \mathbb{Z}_{p}$ is a character whose reduction modulo $p^{n}$ corresponds to the $p^{n}$-th roots of 
\[
    \prod_{\substack{0 \leq a,b <p^{n} \\ p \nmid \gcd(a,b)}}
        \theta(a\omega_{1,n}+b\omega_{2,n}, L)^{a^{m_{1}-1}b^{m_{2}-1}}
\] for every $n \geq 1$ via Kummer theory. Here $\theta(z,L)$ is the fundamental theta function \cite[(2.1)]{Na2} and, $(\omega_{1,n})_{n},(\omega_{2,n})_{n} \in T_{p}(E)$ are defined in Lemma \ref{lmm:basis} (1). 
\end{theorem}

\begin{remark}
(1) Note that $\theta(z,L)$ is defined on $\mathbb{C}$ and is not periodic with respect to $L$. Hence we have to choose a lift of $a\omega_{1,n}+b\omega_{2,n} \in E(\mathbb{C}) \cong \mathbb{C}/L$ to $\mathbb{C}$ to consider the value $\theta(a\omega_{1,n}+b\omega_{2,n}, L)$. However, by \cite[(2.2) Proposition (1)]{Na2}, the values of $\theta(z,L)$ at any such two lifts coincide up to $p$-th power roots of unity. Since $K(E[p^{\infty}])$ contains all $p$-th power roots of unity, the Kummer characters associated to $p$-th power roots of $\theta(z,L)$ at any such two lifts also coincide. This resolves an ambiguity occurring in Theorem \ref{thm:Na}.

(2) Originally, the power series $\alpha_{1,1}$ and its explicit formula is defined and proved by using a certain basis of $\Pi_{1,1}$ coming from that of the topological fundamental group of $X(\mathbb{C})$. However, the proof of Theorem \ref{thm:Na} given in \cite{Na2} works for our basis $\{ x_{1},x_{2} \}$ as it is.
\end{remark}

In the case of $\mathbb{P}^{1}_{\mathbb{Q}}-\{ 0,1, \infty \}$, it is known that coefficients of the universal power series for Jacobi sums correspond to the so-called \emph{Soul\'e characters} $\{ \chi_{m} \colon G_{\mathbb{Q}(\mu_{p^{\infty}})}^{\ab} \to \mathbb{Z}_{p}(m) \}_{m \geq 3, \mathrm{odd}}$ (see Ihara-Kaneko-Yukinari \cite[Theorem $\mathrm{A}_{2}$]{IKY}). Essentially, these characters are  introduced by Soul\'e \cite{So1}, and they enjoy arithmetically interesting properties. We refer the interested reader to the article of Ichimura-Sakaguchi \cite{IS}, for example. Based on this analogy, we define their elliptic analogues as follows:

\begin{definition}[The elliptic Soul\'e character]
For every $\boldsymbol{m}=(m_{1}, m_{2})>\boldsymbol{1}$ such that $\lvert \boldsymbol{m} \rvert$ is even, we call the character $\kappa_{\boldsymbol{m}}$ \emph{the $\boldsymbol{m}$-th elliptic Soul\'e character}.
\end{definition}

Nakamura proved that some linear combinations of the elliptic Soul\'e characters are nontrivial \cite[(3.12)]{Na2}.  In our previous paper \cite{IS1/2}, we conditionally proved that the elliptic Soul\'e characters arising from once-punctured CM elliptic curves are also nontrivial. We briefly summarize our previous results.

Since the power series $\alpha_{1,1}$ is compatible with the action of $\Gal(K(E[p^{\infty}])/K)$, the $\boldsymbol{m}$-th elliptic Soul\'e character $\kappa_{\boldsymbol{m}}$ is a $\Gal(K(E[p^{\infty}])/K)$-equivariant homomorphism
\[
\kappa_{\boldsymbol{m}} \colon G_{K(E[p^{\infty}])}^{\ab} \to \mathbb{Z}_{p}(\boldsymbol{m}).
\] 
Recall that $\Omega$ is defined to be the maximal pro-$p$ extension of $K(p)$ unramified outside $p$.

\begin{theorem}[{\cite[Theorem 1.5 (1) and Lemma 4.1]{IS1/2}}]\label{thm:1}\,
    \begin{enumerate}
        \item The $\boldsymbol{m}$-th elliptic Soul\'e  character $\kappa_{\boldsymbol{m}}$ is trivial unless $\boldsymbol{m}$ is contained in $I$.
        \item When $\boldsymbol{m} \in I$, the character $\kappa_{\boldsymbol{m}}$ factors through the maximal abelian quotient of $\Gal(\Omega/K(p^{\infty}))$. Moreover, it is nontrivial if $H^{2}_{\et}(O_{K}[1/p], \mathbb{Z}_{p}(\boldsymbol{m}))$ is finite.
    \end{enumerate}
\end{theorem}

Here, we write 
\[
H^{2}_{\et}(O_{K}[1/p], \mathbb{Z}_{p}(\boldsymbol{m})) \coloneqq \varprojlim H^{2}_{\et}(O_{K}[1/p], \mathbb{Z}_{p}/p^{n}(\boldsymbol{m}))
\] for the second \'etale cohomology group of the spectrum of the ring of $p$-integers of $K$ with coefficient in $\mathbb{Z}_{p}(\boldsymbol{m})$, whose finiteness is known to follow from a conjecture of Jannsen \cite[Conjecture 1]{Ja}. It vanishes for every $\boldsymbol{m} \in (p-1)\mathbb{Z}_{\geq 1}^{2}$ by \cite[Lemma 4.7]{IS1/2}. It is also finite when the coefficient is the $m$-th Tate twist $\mathbb{Z}_{p}(m)$ for some $m \geq 2$ by a classical result of Soul\'e \cite[page 287, Corollaire]{So3}. However, to the best of the author's knowledge, the finiteness remains an open problem in general. We also have established a sufficient condition that elliptic Soul\'e characters are surjective, based on a relationship between elliptic Soul\'e characters and elliptic units:

\begin{theorem}[{\cite[Theorem 1.5]{IS1/2}}]\label{thm:2}\,
\begin{enumerate}
    \item[(1)] The character $\kappa_{\boldsymbol{m}}$ is not surjective for every $\boldsymbol{m} \in I$ such that $\boldsymbol{m} \geq (2,2)$ and $\boldsymbol{m} \equiv \boldsymbol{1} \mod p-1$. 
    
    \item[(2)] Assume the class number of $K(p)$ is not divisible by $p$ and there are only two primes of $K(p)$ above $p$. Then $\kappa_{\boldsymbol{m}}$ is surjective for every $\boldsymbol{m} \in I$ such that $\boldsymbol{m} \not \equiv \boldsymbol{1} \bmod p-1$. 
\end{enumerate}
\end{theorem}

Although $\kappa_{\boldsymbol{m}}$ is not surjective if $\boldsymbol{m} \geq (2,2)$ and $\boldsymbol{m} \equiv \boldsymbol{1} \mod p-1$, we will see that it is nontrivial under the assumption of Theorem \ref{thm:A}, cf. Corollary \ref{cor:app1}.

Now we discuss a relationship between Conjecture \ref{cnj:DI11} with elliptic Soul\'e characters. One basic observation is the following:

\begin{lemma}\label{lmm:restrict} For $\boldsymbol{m} \in I$, the following assertions hold.
    \begin{enumerate}
        \item[(1)] If $\kappa_{\boldsymbol{m}} \colon F^{1}G_{K} \to \mathbb{Z}_{p}(\boldsymbol{m})$ is nontrivial, then so is its restriction to ${F^{\lvert \boldsymbol{m} \rvert }G_{K}}$.
        \item[(2)] The character $\kappa_{\boldsymbol{m}} \mid_{F^{\lvert \boldsymbol{m} \rvert }G_{K}}$ factors through the $\lvert \boldsymbol{m} \rvert$-th graded quotient $\mathfrak{g}_{\lvert \boldsymbol{m} \rvert}$.
    \end{enumerate}
\end{lemma}
\begin{proof}
We denote $\lvert \boldsymbol{m} \rvert$ by $m$ for simplicity. (1) Suppose that the restricted character $\kappa_{\boldsymbol{m}} \mid_{F^{m}G_{K}}$ is trivial. Then there exists an integer $1 \leq n<m$ satisfying both $\kappa_{\boldsymbol{m}} \mid_{F^{n+1}G_{K}}=0$ and $\kappa_{\boldsymbol{m}} \mid_{F^{n}G_{K}} \neq 0$. In particular, the $n$-th graded quotient $\mathfrak{g}_{n} \otimes \mathbb{Q}_{p}$ has a nontrivial $\chi^{\boldsymbol{m}}$-isotypic component, which is absurd since $n$ is less than $m$. 

(2) It suffices to prove that the character $\kappa_{\boldsymbol{m}}$ vanishes when restricted to $F^{m+2}G_{K}$, since we have $F^{m+1}G_{K}=F^{m+2}G_{K}$ by Proposition \ref{prp:graded} (1). By the construction of the power series $\alpha_{1,1}(\sigma) \in \mathbb{Z}_{p}[[T_{1}, T_{2}]]$, it follows that, for every $\sigma \in F^{m+2}G_{K}$, 
\[
T_{1}\alpha_{1,1}(\sigma)z \in \Pi_{1,1}(m+3)\Pi_{1,1}(2)/[\Pi_{1,1}(2), \Pi_{1,1}(2)] \subset \mathbb{Z}_{p}[[T_{1}, T_{2}]]z.
\] Moreover, the submodule $\Pi_{1,1}(m+3)\Pi_{1,1}(2)/[\Pi_{1,1}(2), \Pi_{1,1}(2)]$ corresponds to $J^{m+1}$, where $J$ is the augmentation ideal of $\mathbb{Z}_{p}[[T_{1}, T_{2}]]$ \cite[(19) on page 67]{Ih1}. Hence the power series $T_{1}\alpha_{1,1}(\sigma)$ is contained in $J^{m+1}$, and we have $\alpha_{1,1}(\sigma) \in J^{m}$. This means that every coefficient of $\alpha_{1,1}(\sigma)$ of a monomial of total degree less than $m$ vanishes. By observing Theorem \ref{thm:Na}, it follows that $\kappa_{\boldsymbol{n}}(\sigma)$ vanishes for every $\boldsymbol{n} \in I$ such that $\lvert \boldsymbol{n} \rvert \leq m-2$ as desired.
\end{proof}

By Lemma \ref{lmm:restrict} (2), we obtain a homomorphism
\[
\mathfrak{g}_{m} \xrightarrow{\oplus \kappa_{\boldsymbol{m}}} \bigoplus_{\boldsymbol{m} \in I, |\boldsymbol{m}|=m} \mathbb{Z}_{p}(\boldsymbol{m})
\] for each even integer $m>2$. Since $\kappa_{\boldsymbol{m}}$ is originally defined on $F^{1}G_{K}$, this homomorphism factors though the $m$-th component $(\mathfrak{g}/[\mathfrak{g},\mathfrak{g}])_{m}$ of the abelianization of $\mathfrak{g}$. The following proposition tells us how to choose a basis of $\mathfrak{g} \otimes \mathbb{Q}_{p}$ in terms of elliptic Soul\'e characters:

\begin{proposition}\label{prp:DI11_Soule}
    Assume that the character $\kappa_{\boldsymbol{m}}$ is nontrivial for every $\boldsymbol{m} \in I$. Then, Conjecture \ref{cnj:DI11} is equivalent to the following statement: For every even $m>2$, the homomorphism
    \[
    (\mathfrak{g}/[\mathfrak{g},\mathfrak{g}])_{m} \otimes \mathbb{Q}_{p} \xrightarrow{\oplus \kappa_{\boldsymbol{m}}}  \bigoplus_{\boldsymbol{m} \in I, |\boldsymbol{m}|=m} \mathbb{Q}_{p}(\boldsymbol{m})
    \] is an isomorphism. Moreover, if we choose $\sigma_{\boldsymbol{m}}$ for each $\boldsymbol{m} \in I$ so that 
\begin{enumerate}
\item[(a)] $\sigma_{\boldsymbol{m}}$ is contained in the $\chi^{\boldsymbol{m}}$-isotypic component of $\mathfrak{g}_{\boldsymbol{\lvert m \rvert}}$, and
\item[(b)]$\kappa_{\boldsymbol{m}}(\sigma_{\boldsymbol{m}}) \neq 0$.
\end{enumerate}
Then the elements $\{ \sigma_{\boldsymbol{m}} \}_{\boldsymbol{m}\in I}$ freely generate $\mathfrak{g} \otimes \mathbb{Q}_{p}$.
\end{proposition}

\begin{proof} Assume Conjecture \ref{cnj:DI11}. The first statement is a consequence of Lemma \ref{lmm:restrict} (1). The second immediately follows from the first, since $\mathfrak{g} \otimes \mathbb{Q}_{p}$ is assumed to be free and the elements $\{ \sigma_{\boldsymbol{m}} \}_{\boldsymbol{m}\in I}$ generate $\mathfrak{g}$ (since they do at the level of the abelianization). 
\end{proof}

\subsection{Main result}\label{2.3}

We turn our attention to field-theoretic properties of the fixed field of the kernel of the pro-$p$ outer Galois representation. First, we study its ramification behavior as an extension over $K(E[p])$:

\begin{lemma}\label{lmm:good}
The field $\bar{K}^{\Ker(\rho_{X,p})}$ is a pro-$p$ extension of $K(E[p])$ unramified outside $p$. 
\end{lemma}

\begin{proof}
    Since the kernel of the homomorphism $\Out(\Pi_{1,1}) \to \Aut(\Pi_{1,1}^{\rm{ab}}/p)$ is pro-$p$, the field $\bar{K}^{\Ker(\rho_{X,p})}$ is a pro-$p$ extension of $K(E[p])$. To prove that the extension is unramified outside $p$, it suffices to show that $X \times_{\Spec(K)} \Spec(K(E[p]))$ or, equivalently, $E \times_{\Spec(K)} \Spec(K(E[p]))$ has good reduction everywhere outside $p$, by virtue of the specialization isomorphism \cite[X, Corollaire 3.9]{SGA1}. Using the N\'eron-Ogg-Shafarevich criterion, it suffices to show that every inertia subgroup $I$ associated with a prime of $K(E[p])$ lying outside $p$ acts trivially on $T_{p}(E)$. Since $E$ has complex multiplication, it has everywhere potentially good reduction. Hence the image of $I$ in $\Aut(T_{p}(E))$ is finite. Moreover, the image of $I$ is contained in $\Ker(\GL(T_{p}(E))\to \GL(E[p]))$, which is a torsion-free pro-$p$ group for every $p \geq 3$. Hence the image of $I$ is trivial as desired.
\end{proof}

Observe that the automorphism group $\Aut_{K}(X)$ defines a subgroup of $\Out(\Pi_{1,1})$. Since it is defined over $K$, it centralizes the image of $\rho_{X,p}$. This puts a strong constraint on the field $\bar{K}^{\Ker(\rho_{X,p})}$ as the following lemma shows:

\begin{lemma}\label{lmm:half}

The field $\bar{K}^{\Ker(\rho_{X,p})}$ is a compositum of $K(E[p])$ and a subfield $\Omega^{\ast} \subset \Omega$ in such a way that $\rho_{X,p}(G_{K(p)})$ splits into the direct product 
\[
\rho_{X,p}(G_{K(p)})= \Gal(\Omega/K(p))=
\Gal(K(E[p])/K(p)) \times \Gal(\Omega^{\ast}/K(p)). 
\] 
\end{lemma}

\begin{proof}
Consider the exact sequence
\[
1 \to \rho_{X, p}(G_{K(E[p])}) \to \rho_{X, p}(G_{K(p)}) \to \Gal(K(E[p])/K(p)) \to 1.
\]
Since $\rho_{X, p}(G_{K(E[p])})$ is a pro-$p$ group and $\Gal(K(E[p])/K(p))$ is a finite prime-to-$p$  group (here we use $p \geq 5$), the exact sequence splits by the Schur-Zassenhaus theorem \cite[5.9, Corollary 1]{Se}. Let us choose an arbitrary section
\[
t \colon \Gal(K(E[p])/K(p)) \to \rho_{X, p}(G_{K(p)}).
\] There is a natural homomorphism $O_{K}^{\times}=\Aut_{K}(X) \to \Out(\Pi_{1,1})$ by functoriality of \'etale fundamental groups, and the image of $\Aut_{K}(X)$ centralizes $\rho_{X,p}(G_{K})$. The $G_{K}$-action on $E[p]$ allows us to identify $\Gal(K(E[p])/K(p))$ with a subgroup of the image of $\Aut_{K}(X)$ in $\Aut_{O_{K}}(E[p])=(O_{K}/p)^{\times}$. Moreover, the automorphism group $\Aut_{K}(X)$ injects into $(O_{K}/p)^{\times}$ since $p$ is prime to the order of $O_{K}^{\times}$. Hence for every  $g \in \Gal(K(E[p])/K(p))$, we can find a unique automorphism $\tilde{g}$ of $X$ such that $t(g)$ and $\tilde{g}$ induce the same element in $\Out(\Pi_{1,1}^{\rm{ab}}/p)=\Aut(E[p])$. Since $t(g)\tilde{g}^{-1}$ has a prime-to-$p$ order and is contained in the pro-$p$ group $\Ker \left( \Out(\Pi_{1,1}) \to \Out(\Pi_{1,1}^{\rm{ab}}/p) \right)$, $t(g)$ must coincide with $\tilde{g}$. This argument shows that the image of $t$ is contained in (the image of) $\Aut_{K}(X)$. Therefore, $t$ induces the decomposition
\[
\rho_{X,p}(G_{K(p)})=\rho_{X,p}(G_{K(E[p])}) \times \Img(t).
\]
Let $\Omega^{\ast}$ be the field corresponding to the kernel of the projection $\rho_{X,p}(G_{K(p)}) \to \rho_{X,p}(G_{K(E[p])})$. It is clear that $\Omega^{\ast}$ is a pro-$p$ extension of $K(p)$ unramified outside $p$, and 
\[
\rho_{X,p}(G_{K(p)})=\Gal(\bar{K}^{\Ker(\rho_{X,p})}/K(p))=\Gal(K(E[p])/K(p)) \times \Gal(\Omega^{\ast}/K(p))
\] is the required decomposition.
\end{proof}

Lemma \ref{lmm:half} allows us to formulate the following question, which is an analogue of Anderson-Ihara's question in Introduction:
\begin{center}
\emph{Is the field $\Omega^{\ast}$ is equal to $\Omega$?}
\end{center}
In analogy with Sharifi's result, one may wonder if the question has an affirmative answer, assuming Conjecture \ref{cnj:DI11} and a certain arithmetic assumption on $p$ is satisfied. This is exactly the main result of the present paper:

\begin{theorem}\label{thm:main}
Let $p \geq 5$ be a prime which splits in $K$, and assume that the following conditions hold:
\begin{enumerate}
    \item[(1)] the class number of $K(p)$ is not divisible by $p$,
    \item[(2)] there are exactly two primes of $K(p^{2})$ above $p$, and
    \item[(3)] Conjecture \ref{cnj:DI11} holds. 
\end{enumerate}
Then $\Omega^{\ast}$ coincides with $\Omega$. Consequently, we have $\bar{K}^{\Ker(\rho_{X,p})}=K(E[p]) \cdot \Omega$.
\end{theorem}

The proof of Theorem \ref{thm:main} will be given at the end of Section \ref{4}.

\begin{remark}\label{rmk:secondcond}
    We keep the same notation as in Theorem \ref{thm:main}, and write $p=\pi\bar{\pi}$ for a prime element $\pi \in K$ and its conjugate $\bar{\pi}$. Then the following statements are equivalent:
    \begin{enumerate}
        \item There are exactly two primes of $K(p^{2})$ above $p$.
        \item $\Gal(K(p^{2})/K)$ coincides with its decomposition subgroup at $\p$.    
        \item The image of $\pi$ in $(O_{K}/\bar{\p}^{2})^{\times}/O_{K}^{\times}$ is a generator of this cyclic group.
        \item $\Gal(K(p^{\infty})/K)$ coincides with its decomposition subgroup at $\p$.
        \item The image of $\pi$ in $O_{K_{\bar{\p}}}^{\times}/O_{K}^{\times}$ is a (topological) generator of this procyclic group.
    \end{enumerate}
\end{remark}

Observe that we do not use the assumption that $p$ splits in $K$ to prove Lemma \ref{lmm:half} (though we use the fact that $p$ is prime to the order of $O_{K}^{\times}$). We expect an analogue of Anderson-Ihara's question to have an affirmative answer even if $p$ does not split. For example, we can show the following unconditional result, whose proof relies on the case of genus zero (this  result is not used in the rest of the paper):

\begin{proposition}\label{prp:ramify}
    Suppose
    \[
    K  \in \{ \mathbb{Q}(\sqrt{-7}), \mathbb{Q}(\sqrt{-11}),  \mathbb{Q}(\sqrt{-19}), \mathbb{Q}(\sqrt{-43}), \mathbb{Q}(\sqrt{-167}) 
    \}
    \] and let $p_{K}$ be a unique rational prime which ramifies in $K$. Then we have  $\bar{K}^{\Ker(\rho_{X,p_{K}})}=K(E[p_{K}]) \cdot \Omega$.
\end{proposition}
\begin{proof} 
    By our assumption, we have
    \[
    p_{K} \in \{ 7,11,19,43,163 \}
    \] and hence $p_{K}$ is odd and regular. Moreover, it is easy to observe that the mod-$p$ ray class field $K(p_{K})$ is a $p_{K}$-extension of $\mathbb{Q}(\mu_{p_{K}})=K(\mu_{p_{K}})$ unramified outside $p_{K}$, since it contains $K(\mu_{p_{K}})$ and  $[K(p):K(\mu_{p})]=p$. On the other hand, by ``Oda's prediction'' \cite[Theorem 3.6]{Ta12} established by Takao, it follows that
    \[
     \bar{\mathbb{Q}}^{\Ker(\rho_{\mathbb{P}^{1}_{\mathbb{Q}} \setminus \{ 0,1,\infty\}, p_{K}})}
     \subset
    \bar{K}^{\Ker(\rho_{X,p_{K}})}.
    \] However, the former field coincides with $\Omega^{\mathrm{cyc}}$ by Sharifi's theorem \cite[Theorem 1.1]{Sh}, together with the resolution of the Deligne-Ihara conjecture by Hain-Matsumoto \cite{HM2} and Brown \cite{Br}. Since $\Omega^{\mathrm{cyc}}$ is also the maximal pro-$p_{K}$ extension of $K(p_{K})$ unramified outside $p_{K}$, it coincides with $\Omega$. This concludes the proof since the proof of Lemma \ref{lmm:half} works even when $p \geq 5$ ramifies in $K$.
\end{proof}

\begin{remark}
    Proposition \ref{prp:ramify} even holds for $K=\mathbb{Q}(\sqrt{-3})$ (and hence $p_{K}=3$); In fact, since $E$ is then isomorphic to the Fermat curve of degree $3$ over $K(E[3])$, we may assume that $E$ is the Fermat curve of degree $3$. Then it is easy to observe $K=K(3)=K(E[3])$. B Lemma \ref{lmm:good} and a similar argument as that of Proposition \ref{prp:ramify}, we have
    $
        \Omega^{\mathrm{cyc}}=\Omega \subset \bar{K}^{\Ker(\rho_{X,3})} \subset \Omega
    $ as desired.
\end{remark}

\section{Two-variable filtrations on profinite groups}\label{3}

In this section, we define two-variable filtrations on various profinite groups, e.g. the pro-$p$ fundamental groups of once-punctured elliptic curves, subgroups of the pro-$p$ mapping class group of type $(1,1)$ and Galois groups, and establish some fundamental properties. Throughout this section, let $\Pi$ denote a free pro-$p$ group of rank two with basis $\{ x, y \}$ and define $z \coloneqq [y,x]$.

\subsection{Two-variable filtration on free pro-$p$ group of rank two}\label{3.1}

First, we define a two-variable filtration on $\Pi$, which is an analogue of the descending central series.

\begin{definition}\label{dfn:bifilt}
    For $\boldsymbol{m} \in \mathbb{Z}_{\geq 0}^{2} \setminus \{ (0,0) \}$, we inductively define a normal subgroup $\Pi(\boldsymbol{m})$ of $\Pi$ as follows:
\begin{enumerate}
\item[(1)] Let $\Pi(1, 0)$ (resp. $\Pi(0, 1)$) be the normal closure of $x$ (resp. $y$) in $\Pi$. 
\item[(2)] For $\boldsymbol{m}=(m_{1}, m_{2}) \in \mathbb{Z}_{\geq 0}^{2}$ with $|\boldsymbol{m}| \geq 2$, we define the subgroup $\Pi(\boldsymbol{m}) \subset \Pi$ by
\[ 
\Pi(\boldsymbol{m}) \coloneqq \langle [\Pi(\boldsymbol{m}'), \Pi(\boldsymbol{m}'')] \mid \boldsymbol{m}'+\boldsymbol{m}''=\boldsymbol{m} \text{ where $\boldsymbol{m}', \boldsymbol{m}'' \in \mathbb{Z}_{\geq 0}^{2} \setminus \{ (0,0) \}$} \rangle.
\] Since $\Pi(1,0)$ and $\Pi(0,1)$ are normal subgroups, the subgroup $\Pi(\boldsymbol{m})$ is also normal.
\end{enumerate}
\end{definition}

The definition of the two-variable filtration depends on the choice of the basis $\{x,y\}$ of $\Pi$. More precisely, $\Pi(1,0)$ (resp. $\Pi(0,1)$) depends on the image $\bar{y}$ of $y$ (resp. $\bar{x}$ of $x$) in the maximal abelian quotient $\Pi^{\ab}$ of $\Pi$.  

\begin{example}
By definition, $\Pi(1,1)$ coincides with the commutator subgroup $\Pi(2)$.
\end{example}

We have the following inclusions and equalities:

\begin{lemma}\label{lmm:incl}
Let $\boldsymbol{m}=(m_{1}, m_{2}), \boldsymbol{n}=(n_{1}, n_{2}) \in \mathbb{Z}_{\geq 0}^{2} \setminus \{ (0,0) \}$ and $m \geq 2$.

(1) The inclusion $\Pi(\boldsymbol{m}) \subset \Pi(|\boldsymbol{m}|)$ holds.

(2) The equality $\Pi(m,0)=\Pi(m,1)$ holds. Similarly, we have $\Pi(0,m)=\Pi(1, m)$.

(3) If  $\boldsymbol{m} \geq \boldsymbol{n}$, then the inclusion $\Pi(\boldsymbol{m}) \subset \Pi(\boldsymbol{n})$ holds.
\end{lemma}
\begin{proof} 

(1) The assertion immediately follows by induction on $\lvert \boldsymbol{m} \rvert$.

(2) The inclusion $\Pi(m,1) \subset \Pi(m,0)$ follows by induction on $m$. To prove the opposite inclusion by induction on $m$, it suffices to prove the assertion for $m=2$. Note that the group $\Pi(2,0)/\Pi(2,1)$ is normally generated by the image of the commutator map 
\[
\Pi(1,0)/\Pi(\boldsymbol{1}) \times \Pi(1,0)/\Pi(\boldsymbol{1}) \to \Pi(2,0)/\Pi(2,1).
\]
However, since $\Pi(1,0)/\Pi(\boldsymbol{1})$ is generated by the image of $x$, the image of the commutator map is trivial. Hence the equality $\Pi(2,0)=\Pi(2,1)$ follows.

(3)  We prove the assertion by induction on $|\boldsymbol{m}|+|\boldsymbol{n}|$. If  $|\boldsymbol{m}|+|\boldsymbol{n}|=2$, then the assertion is clear. Let us assume that $|\boldsymbol{m}|+|\boldsymbol{n}|>2$. Since $\Pi(\boldsymbol{m})$ is normally generated by $[\Pi(\boldsymbol{m}'), \Pi(\boldsymbol{m}'')]$ with $\boldsymbol{m}'+\boldsymbol{m}''=\boldsymbol{m}$, the assertion follows if there exist $\boldsymbol{n}'$ and $\boldsymbol{n}''$ such that $\boldsymbol{m}' \geq \boldsymbol{n}'$,  $\boldsymbol{m}'' \geq \boldsymbol{n}''$ and $\boldsymbol{n}'+\boldsymbol{n}''=\boldsymbol{n}$, by induction hypothesis. Such a pair $(\boldsymbol{n}', \boldsymbol{n}'')$ clearly exists unless $(m_{1}, n_{1})$ or $(m_{2}, n_{2})$ is equal to $(1,0)$. However, the assertion in this exceptional case also follows by using (2).
\end{proof}

\begin{definition}

For $\mathbb{Z}_{\geq 0}^{2} \setminus \{ (0,0) \}$, we define graded quotients $\Gr^{\boldsymbol{m}}_{1} \Pi$ and $\Gr^{\boldsymbol{m}}_{2} \Pi$ of $\Pi$ as 
\[
\Gr^{\boldsymbol{m}}_{1} \Pi \coloneqq \Pi(\boldsymbol{m})/\Pi(\boldsymbol{m}+(1,0)) \quad \text{and} \quad
\Gr^{\boldsymbol{m}}_{2} \Pi \coloneqq \Pi(\boldsymbol{m})/\Pi(\boldsymbol{m}+(0,1)).
\] 
\end{definition}

Note that $\Gr^{\boldsymbol{m}}_{1} \Pi$ (resp. $\Gr^{\boldsymbol{m}}_{2} \Pi $) is a  $\mathbb{Z}_{p}[[\Pi/\Pi(1,0)]]$ (resp. $\mathbb{Z}_{p}[[\Pi/\Pi(0,1)]]$)-module where the group $\Pi/\Pi(1,0)$ (resp. $\Pi/\Pi(0,1)$) acts by conjugation.

\begin{example}
(1) $\Gr^{(1,0)}_{1} \Pi = \Pi(1,0)/\Pi(2,0)$ is a free $\mathbb{Z}_{p}[[\Pi/\Pi(1,0)]]$-module of rank one generated by $x$, cf. \cite[Theorem 2.2]{Ih1}. Similarly, $\Gr^{(0,1)}_{2} \Pi$ is a free $\mathbb{Z}_{p}[[\Pi/\Pi(0,1)]]$-module of rank one generated by $y$.

(2) $\Gr^{(1,0)}_{2} \Pi = \Pi(1,0)/\Pi(\boldsymbol{1})$ is a free $\mathbb{Z}_{p}$-module of rank one generated by the image of $x$ on which $\Pi/\Pi(1,0)$ acts trivially. Similarly,  $\Gr^{(0,1)}_{1} \Pi$ is generated by $y$ on which $\Pi/\Pi(0,1)$ acts trivially. On the other hand, by Lemma \ref{lmm:incl}(2), we have $\Gr^{(m,0)}_{2} \Pi = \Gr^{(0,m)}_{1} \Pi = 0$ for every $m \geq 2$.
\end{example}

\subsection{Two-variable filtration on pro-$p$ mapping class group}

In this subsection, we define a two-variable filtration on subgroups of the automorphism group of $\Pi$. The aim of this subsection is to prove Proposition \ref{prp:act_gamma_3}, which computes the action on the graded quotients associated to the two-variable filtration. We note that the content of this subsection may be regarded as a two-variable variant of a part of a work of Nakamura-Tsunogai \cite{NT}, studying the structure of graded quotients of pro-$p$ mapping class groups with respect to weight filtrations.

\begin{definition}
We define two subgroups $\tilde{\Gamma}$ and $\Gamma^{\dag}$ of $\Aut(\Pi)$ as 
\[
\tilde{\Gamma} \coloneqq \left\{ f \in \Aut(\Pi) 
\, \middle| \,
\begin{aligned}
&\text{$\bar{f}$ preserves $\langle \bar{x} \rangle$ and $\langle \bar{y} \rangle$ respectively, and}\\
&\text{$f$ preserves the conjugacy class of $\langle z \rangle$}
\end{aligned}
 \right\} \text{and}
\] 
\[
\Gamma^{\dag} \coloneqq \{ f \in \tilde{\Gamma} \mid \text{$f$ preserves $\langle z \rangle$} \}, 
\] where $\bar{f} \in \Aut(\Pi^{\ab})$ denotes the automorphism of $\Pi^{\ab}$ induced by $f$. We set $\Gamma \coloneqq \tilde{\Gamma}/\Inn(\Pi)$, which is a subgroup of $\Out(\Pi)$.
\end{definition}
 
One can easily observe that the subgroup $\Pi(\boldsymbol{m})$ defined in the last subsection is preserved under the action of $\tilde{\Gamma}$. Moreover, by the definition of $\tilde{\Gamma}$, there exist two characters\footnote{This notation is ambiguous since we already use the symbols $\chi_{1}$ and $\chi_{2}$ to indicate the characters $G_{K} \to \mathbb{Z}_{p}^{\times}$. However, if we identify $\Pi$ with the fundamental group $\Pi_{1,1}$ via $x=x_{1}$ and $y=x_{2}$, then the image of the pro-$p$ outer Galois representation $\rho_{X,p}$ is contained in $\Gamma$, and these notations become compatible.}
\[
\chi_{1} \colon \tilde{\Gamma} \to \Aut(\langle \bar{x} \rangle)=\mathbb{Z}_{p}^{\times}
\quad \text{and} \quad
 \chi_{2} \colon \tilde{\Gamma} \to \Aut(\langle \bar{y} \rangle)=\mathbb{Z}_{p}^{\times}
\] which are characterized by the equalities $\bar{f}(\bar{x})=\bar{x}^{\chi_{1}(f)}$ and $\bar{f}(\bar{y})=\bar{y}^{\chi_{2}(f)}$. These characters factor through the quotient $\Gamma$. There is also a character $\Gamma^{\dag} \to \Aut(\langle z \rangle)=\mathbb{Z}_{p}^{\times}$ by the definition of $\Gamma^{\dag}$, but this coincides with the product $\chi_{1}\chi_{2}$, since the commutator map $\Pi/\Pi(2) \times \Pi/\Pi(2) \to \Pi(2)/\Pi(3)$ is bilinear.

\begin{remark}\label{rmk:relation}
   If we identify $\Pi$ with the pro-$p$ fundamental group $\Pi_{1,1}$ of the once-punctured CM elliptic curve $X$ via $x=x_{1}$ and $y=x_{2}$, then $\tilde{\Gamma}$ (resp. $\Gamma^{\dag}$, $\Gamma$) is a subgroup of $\tilde{\Gamma}_{1,1}$ (resp. $\Gamma_{1,1}^{\dag}$, $\Gamma_{1,1}$) introduced in Section \ref{2.1}. Moreover, through this identification, the $G_{K}$-action preserves $\langle \bar{x} \rangle$ and $\langle \bar{y} \rangle$, since they are isomorphic to the $\p$-adic and the $\bar{\p}$-adic Tate module of $E$, respectively. In particular, the image of the pro-$p$ outer Galois representation $\rho_{X,p}:G_{K} \to \Gamma_{1,1}$ is contained in $\Gamma$.
\end{remark}

Recall that, in Section \ref{2.1}, the weight filtration $\{ F^{m}\tilde{\Gamma} \}_{m \geq 1}$ on $\tilde{\Gamma}$ is defined by 
\[
F^{m}\tilde{\Gamma} \coloneqq \Ker \left( \tilde{\Gamma} \to \Aut(\Pi/\Pi(m+1)) \right)
\] for $m \geq 1$. By replacing the descending central series with the two-variable filtration $\{ \Pi(\boldsymbol{m}) \}_{\boldsymbol{m}}$ on $\Pi$, we define a two-variable analogue of the weight filtration as follows.

\begin{definition}\label{dfn:gamma}
For every $\boldsymbol{m} \in \mathbb{Z}_{\geq 0}^{2} \setminus \{ (0,0) \}$, we define subgroups $F^{\boldsymbol{m}}\tilde{\Gamma}$, $F^{\boldsymbol{m}}\Gamma^{\dag}$ and $F^{\boldsymbol{m}}\Gamma$ by
\begin{align*}
& F^{\boldsymbol{m}}\tilde{\Gamma} \coloneqq \Ker \left( \tilde{\Gamma} \to \prod_{\boldsymbol{k} \in \{(0,1), (1,0)\} }\Aut \left( \Pi(\boldsymbol{k})/\Pi(\boldsymbol{m}+\boldsymbol{k}) \right) \right), \\
& F^{\boldsymbol{m}}\Gamma^{\dag} \coloneqq F^{\boldsymbol{m}}\Gamma \cap \Gamma^{\dag},
\quad \text{and} \quad F^{\boldsymbol{m}}\Gamma \coloneqq \Img(F^{\boldsymbol{m}}\tilde{\Gamma} \to \Gamma).
\end{align*}

Note that, since $[\Pi(\boldsymbol{m}), \Pi(\boldsymbol{k})] \subset \Pi(\boldsymbol{m}+\boldsymbol{k})$ for $\boldsymbol{k} \in \{(1,0), (0,1)\}$, the inner automorphism group $\Inn_{\Pi(\boldsymbol{m})}(\Pi)$ of $\Pi$ induced by elements of $\Pi(\boldsymbol{m})$ is contained in $F^{\boldsymbol{m}}\tilde{\Gamma}$. 
\end{definition}

Note that the subgroup $F^{\boldsymbol{m}}\tilde{\Gamma}$ acts trivially on the quotient $\Pi/ \left( \Pi(\boldsymbol{m}) \cap \Pi(|\boldsymbol{m}|+1) \right)$ by definition and Lemma \ref{lmm:incl}, hence on $\Pi/ \Pi(|\boldsymbol{m}|+1)$. We therefore obtain the following comparison with the usual weight filtration:
    \[
        F^{\boldsymbol{m}}\tilde{\Gamma} \subset F^{|\boldsymbol{m}|}\tilde{\Gamma}.
    \] The point of Definition \ref{dfn:gamma} is to impose conditions on the images of $x$ and $y$ separately, i.e. $f(x)x^{-1} \in \Pi(\boldsymbol{m}+(1,0))$ and $f(y)y^{-1} \in \Pi(\boldsymbol{m}+(0,1))$ for each $f \in F^{\boldsymbol{m}}\tilde{\Gamma}$. In fact, the following lemma characterizes elements of $F^{\boldsymbol{m}}\tilde{\Gamma}$:

\begin{lemma}\label{lmm:contain}
Let $\boldsymbol{m} \in \mathbb{Z}_{\geq 0}^{2} \setminus \{ (0,0) \}$ and $f \in \tilde{\Gamma}$. Then the automorphism $f$ is contained in $F^{\boldsymbol{m}}\tilde{\Gamma}$ if and only if $f(x)x^{-1} \in \Pi(\boldsymbol{m}+(1,0))$ and $f(y)y^{-1} \in \Pi(\boldsymbol{m}+(0,1))$. 
\end{lemma}
\begin{proof}
It suffices to prove only the ``if'' part of the assertion. Since the group $\Pi(1,0)$ is generated by $\{ y^nxy^{-n} \}_{n \geq 0}$, it suffices to prove that $f(y^{n}xy^{-n})y^{n}x^{-1}y^{-n} \in \Pi(\boldsymbol{m}+(1,0))$ for every $n \geq 0$. We compute this term as follows.
\begin{align*}
f(y^{n}xy^{-n})y^{n}x^{-1}y^{-n} &= f(y)^{n} f(x) f(y)^{-n}y^{n} x^{-1} y^{-n} \\
&\equiv f(y)^{n} f(y)^{-n}y^{n} f(x) x^{-1} y^{-n} \bmod \Pi(\boldsymbol{m}+(1,0)) \\
&= y^{n}  f(x) x^{-1} y^{-n} \equiv 1.
\end{align*}
Here, we use $[f(y)^{-n}y^{n}, f(x)] \in [\Pi(\boldsymbol{m}+(0,1)), \Pi(1,0)] \subset \Pi(\boldsymbol{m}+\boldsymbol{1}) \subset \Pi(\boldsymbol{m}+(1,0))$ to establish the first congruence. A similar computation shows that $f$ also acts trivially on $\Pi(0,1)/\Pi(\boldsymbol{m}+(0,1))$. 
\end{proof}

Moreover, we have the following lemma:

\begin{lemma}\label{lmm:trivial_action}
Let $\boldsymbol{m} \in \mathbb{Z}_{\geq 0}^{2} \setminus \{ (0,0) \}$. Then $F^{\boldsymbol{m}}\tilde{\Gamma}$ acts trivially on $\Pi(\boldsymbol{n})/\Pi(\boldsymbol{n}+\boldsymbol{m})$ for every $\boldsymbol{n} \in \mathbb{Z}_{\geq 0}^{2} \setminus \{ (0,0) \}$. 
\end{lemma}
\begin{proof}
We prove the assertion by induction on $\lvert \boldsymbol{n} \rvert$. First, if $\lvert \boldsymbol{n} \rvert=1$, the assertion follows by the definition of $F^{\boldsymbol{m}}\tilde{\Gamma}$. In general, since $\Pi(\boldsymbol{n})$ is generated by subgroups $[\Pi(\boldsymbol{n}'), \Pi(\boldsymbol{n}'')]$ such that $\boldsymbol{n}'+\boldsymbol{n}''=\boldsymbol{n}$ and $0<|\boldsymbol{n}'|, |\boldsymbol{n}''|<|\boldsymbol{n}|$, the assertion follows by the fact that the commutator map
\[
\Pi(\boldsymbol{n}')/\Pi(\boldsymbol{n}'+\boldsymbol{m}) \times \Pi(\boldsymbol{n}'')/\Pi(\boldsymbol{n}'+\boldsymbol{m}) \to \Pi(\boldsymbol{n})/\Pi(\boldsymbol{n}+\boldsymbol{m})
\] is compatible with the action of $\tilde{\Gamma}$.
\end{proof}

Similar to the case of $\Pi$, we consider two-variable graded quotients as follows.

\begin{definition} For $\boldsymbol{m} \in \mathbb{Z}_{\geq 0}^{2} \setminus \{ (0,0) \}$, we define $\Gr^{\boldsymbol{m}}_{1} \tilde{\Gamma}$ and $\Gr^{\boldsymbol{m}}_{2} \tilde{\Gamma}$ by 
\[
\Gr^{\boldsymbol{m}}_{1} \tilde{\Gamma} \coloneqq F^{\boldsymbol{m}}\tilde{\Gamma}/F^{\boldsymbol{m}+(1,0)}\tilde{\Gamma} \quad \text{and} \quad \Gr^{\boldsymbol{m}}_{2} \tilde{\Gamma} \coloneqq F^{\boldsymbol{m}}\tilde{\Gamma}/F^{\boldsymbol{m}+(0,1)}\tilde{\Gamma}.
\] Similarly, we define $\Gr^{\boldsymbol{m}}_{1} \Gamma$ and $\Gr^{\boldsymbol{m}}_{2} \Gamma$ by
\[
    \Gr^{\boldsymbol{m}}_{1} \Gamma \coloneqq F^{\boldsymbol{m}}\Gamma/F^{\boldsymbol{m}+(1,0)}\Gamma \quad \text{and} \quad \Gr^{\boldsymbol{m}}_{2} \Gamma \coloneqq F^{\boldsymbol{m}}\Gamma/F^{\boldsymbol{m}+(0,1)}\Gamma.
\] 
\end{definition}

Note that, by definition, we have natural surjective homomorphisms
\[
 \Gr^{\boldsymbol{m}}_{1} \tilde{\Gamma} \to \Gr^{\boldsymbol{m}}_{1} \Gamma \quad \text{and} \quad \Gr^{\boldsymbol{m}}_{2} \tilde{\Gamma} \to \Gr^{\boldsymbol{m}}_{2} \Gamma.
\]
The following homomorphisms are useful in studying these quotients.

\begin{definition}
For $\boldsymbol{m} \in \mathbb{Z}_{\geq 0}^{2} \setminus \{ (0,0) \}$, we define homomorphisms $i_{\boldsymbol{m},1}$ and $i_{\boldsymbol{m},2}$ by 
\begin{align*}
&i_{\boldsymbol{m},1} \colon \Gr^{\boldsymbol{m}}_{1} \tilde{\Gamma} \to \Gr^{\boldsymbol{m}+(1,0)}_{1} \Pi \oplus \Gr^{\boldsymbol{m}+(0,1)}_{1} \Pi \colon f \mapsto (f(x)x^{-1}, f(y)y^{-1}) \quad \text{and} \\
&i_{\boldsymbol{m},2} \colon \Gr^{\boldsymbol{m}}_{2} \tilde{\Gamma} \to \Gr^{\boldsymbol{m}+(1,0)}_{2} \Pi \oplus \Gr^{\boldsymbol{m}+(0,1)}_{2} \Pi \colon f \mapsto (f(x)x^{-1}, f(y)y^{-1}).
\end{align*}
\end{definition}

Note that Lemma \ref{lmm:contain} implies that both $i_{\boldsymbol{m},1}$ and $i_{\boldsymbol{m}, 2}$ are injective. 

\begin{remark}
    Two homomorphisms $i_{\boldsymbol{m},1}$ and $i_{\boldsymbol{m},2}$ defined above are analogues of the homomorphism $\tilde{\iota}_{m}$ defined in \cite[(1.11.2)]{NT}, which relates the (one-variable) graded quotients of the pro-$p$ mapping class groups with those of the pro-$p$ fundamental groups.
\end{remark}

Write
\[
\tilde{\Gamma}_{1} \coloneqq \{ \gamma \in \tilde{\Gamma} \mid \gamma(y)y^{-1} \in \Pi(2) \} \quad \text{and} \quad 
\tilde{\Gamma}_{2} \coloneqq \{ \gamma \in \tilde{\Gamma} \mid \gamma(x)x^{-1} \in \Pi(2) \}.
\] Moreover, we define  
\begin{align*}
& \Gamma^{\dag}_{1} \coloneqq \tilde{\Gamma}_{1} \cap \Gamma^{\dag}, \quad \Gamma_{1} \coloneqq \Img(\tilde{\Gamma}_{1} \to \Gamma), \\
& \Gamma^{\dag}_{2} \coloneqq \tilde{\Gamma}_{2} \cap \Gamma^{\dag}, \quad \Gamma_{2} \coloneqq \Img(\tilde{\Gamma}_{2} \to \Gamma).
\end{align*} Then the action of $\tilde{\Gamma}_{1}/F^{(1,0)}\tilde{\Gamma}$ (resp. $\tilde{\Gamma}_{2}/F^{(0,1)}\tilde{\Gamma}$) on $\Gr^{\boldsymbol{m}}_{1} \Pi $ (resp. $\Gr^{\boldsymbol{m}}_{2} \Pi$) commutes with the action of $\Pi/\Pi(1,0)$ (resp. $\Pi/\Pi(0,1)$). In the following, we study the action of $\Gamma^{\dag}_{1}/F^{(1,0)}\Gamma^{\dag}$ (resp.  $\Gamma^{\dag}_{2}/F^{(0,1)}\Gamma^{\dag}$) on $\Gr^{\boldsymbol{m}}_{1}$ (resp. $\Gr^{\boldsymbol{m}}_{2}$) of various groups.  

\begin{lemma}\label{lmm:act_gamma}
    The element $\gamma(x)x^{-\chi_{1}(\gamma)}$ is contained in $\Pi(2,0)$ for every $\gamma \in \Gamma^{\dagger}_{1}$. Similarly, $\gamma(y)y^{-\chi_{2}(\gamma)}$ is contained in $\Pi(0,2)$ for every $\gamma \in \Gamma^{\dag}_{2}$.
\end{lemma}

\begin{proof}
We only prove the first assertion. Since $\Gamma^{\dag}_{1}$ acts $\mathbb{Z}_{p}[[\Pi/\Pi(1,0)]]$-linearly on $\Pi(1,0)/\Pi(2,0)$, which is a free $\mathbb{Z}_{p}[[\Pi/\Pi(1,0)]]$-module generated by $x$,
 the action of $\gamma$ is given by a scalar multiplication by an element of $\mathbb{Z}_{p}[[\Pi/\Pi(1,0)]]^{\times}$. However, since we have $0 \neq z=[y,x]=(y-1)x$ in $\Pi(1,0)/\Pi(2,0)$ and $\gamma(z)=z^{\chi_{1}(\gamma)}$, such a scalar must coincide with $\chi_{1}(\gamma)$ as desired.
 \end{proof}

\begin{lemma}\label{lmm:act_gamma_2}
Let $\boldsymbol{m}=(m_{1}, m_{2}) \in \mathbb{Z}_{\geq 0}^{2} \setminus \{ (0,0) \}$. Then the action of $\Gamma^{\dag}_{1}/F^{(1,0)}\Gamma^{\dag}$ on  $\Gr^{\boldsymbol{m}}_{1} \Pi$ is given by the character $\chi_{1}^{m_{1}}$. Similarly, the action of $\Gamma^{\dag}_{2}/F^{(0,1)}\Gamma^{\dag}$ on $\Gr^{\boldsymbol{m}}_{2} \Pi$ is given by $\chi_{2}^{m_{2}}$. 
\end{lemma}

\begin{proof}
Fix an arbitrary element $\gamma \in \Gamma^{\dag}_{1}$. We show that the action of $\gamma$ on $\Gr^{\boldsymbol{m}}_{1} \Pi$ is given by multiplication by $\chi_{1}^{m_{1}}(\gamma)$ by induction on $|\boldsymbol{m}|$. If $|\boldsymbol{m}|=1$, the assertion follows from (the proof of) Lemma \ref{lmm:act_gamma}. In general, the graded quotient $\Gr^{\boldsymbol{m}}_{1} \Pi$ is generated (as a $\mathbb{Z}_{p}[[\Pi/\Pi(1,0)]]$-module) by the image of commutator maps
\[
[ \; \cdot \; , \; \cdot \; ] \colon \Gr^{\boldsymbol{m}'}_{1} \Pi \times \Gr^{\boldsymbol{m}''}_{1} \Pi \to \Gr^{\boldsymbol{m}}_{1} \Pi,
\] where $\boldsymbol{m}'=(m_{1}',m_{2}')$ and $\boldsymbol{m}''=(m_{1}'',m_{2}'')$ satisfy $\boldsymbol{m}'+\boldsymbol{m}''=\boldsymbol{m}$. Since this pairing is bilinear, for every $(\tau',\tau'') \in \Gr^{\boldsymbol{m}'}_{1} \Pi \times \Gr^{\boldsymbol{m}''}_{1} \Pi$, we have
\[
\gamma([\tau',\tau''])=[\gamma(\tau'),\gamma(\tau'')]=[\chi_{1}^{m'_{1}}(f)\tau', \chi_{1}^{m''_{1}}(\gamma)\tau'']=\chi_{1}^{m_{1}}(\gamma)[\tau',\tau'']
\] by using the induction hypothesis. This concludes the proof.
\end{proof}

We now prove the main result in this subsection.

\begin{proposition}\label{prp:act_gamma_3}
For every $\boldsymbol{m}=(m_{1},m_{2}) \in \mathbb{Z}_{\geq 0}^{2} \setminus \{ (0,0) \}$, the homomorphism 
\[
i_{\boldsymbol{m},1} \colon \Gr^{\boldsymbol{m}}_{1} \tilde{\Gamma} \to \Gr^{\boldsymbol{m}+(1,0)}_{1} \Pi(\chi_{1}^{-1}) \oplus \Gr^{\boldsymbol{m}+(0,1)}_{1} \Pi 
\] is compatible with the action of $\Gamma^{\dag}_{1}/F^{(1,0)}\Gamma^{\dag}$, where $\Gr^{\boldsymbol{m}+(1,0)}_{1} \Pi (\chi_{1}^{-1})$ denotes the $\chi_{1}^{-1}$-twist of $\Gr^{\boldsymbol{m}+(1,0)}_{1} \Pi$. Similarly, the homomorphism
\[
i_{\boldsymbol{m},2} \colon \Gr^{\boldsymbol{m}}_{2} \tilde{\Gamma}  \to \Gr^{\boldsymbol{m}+(1,0)}_{2} \Pi \oplus \Gr^{\boldsymbol{m}+(0,1)}_{2} \Pi(\chi_{2}^{-1})
\] is compatible with the action of $\Gamma^{\dag}_{2}/F^{(0,1)}\Gamma^{\dag}$. In particular, the action of $\Gamma^{\dag}_{1}/F^{(1,0)}\Gamma^{\dag}$ (resp. $\Gamma^{\dag}_{2}/F^{(0,1)}\Gamma^{\dag}$) on $\Gr^{\boldsymbol{m}}_{1} \tilde{\Gamma}$ (resp. $\Gr^{\boldsymbol{m}}_{2} \tilde{\Gamma}$) is given by $\chi_{1}^{m_{1}}$ (resp. $\chi_{2}^{m_{2}}$).
\end{proposition}

\begin{proof}
Before proving the assertion, we note that the group $\Gamma^{\dag}_{1}/F^{(1,0)}\Gamma^{\dag}$ acts on the graded quotients $\Gr^{\boldsymbol{m}}_{1}\Pi$ by  $\chi_{1}^{m_{1}}$ by Lemma \ref{lmm:act_gamma_2}. Take an arbitrary element $\gamma$ of $\Gamma^{\dag}_{1}$. We prove the first assertion by computing the image of $\gamma \cdot f$ under the map $i_{\boldsymbol{m}, 1}$ for an arbitrary $f \in F^{\boldsymbol{m}}\tilde{\Gamma}$.

First, we have
\begin{align*}
(\gamma \cdot f)(y)y^{-1} 
&= (\gamma f \gamma^{-1})(y)y^{-1} \\
&= \gamma(f(\gamma^{-1}(y))\gamma^{-1}(y^{-1})).
\end{align*} 
Since $\alpha \coloneqq y^{-1}\gamma^{-1}(y)$ is contained in $\Pi(2)=\Pi(\boldsymbol{1})$, we have
\[
\gamma(f(\gamma^{-1}(y))\gamma^{-1}(y^{-1}))=\gamma(f(y\alpha)\alpha^{-1}y^{-1}) \equiv \gamma(f(y)y^{-1}) \bmod \Pi(\boldsymbol{m}+\boldsymbol{1}).
\] Here, we use $f(\alpha)\alpha^{-1} \in \Pi(\boldsymbol{m}+\boldsymbol{1})$, which follows from Lemma \ref{lmm:trivial_action}, to deduce the last congruence. By Lemma \ref{lmm:act_gamma_2}, the last term coincides with $\chi_{1}^{m_{1}}(\gamma)(f(y)y^{-1})$ as desired. Secondly, we compute $(\gamma \cdot f)(x)x^{-1}$ as
\begin{align*}
(\gamma \cdot f)(x)x^{-1} 
&= (\gamma f \gamma^{-1})(x)x^{-1} \\
&= \gamma(f(\gamma^{-1}(x))\gamma^{-1}(x^{-1})).
\end{align*} By Lemma \ref{lmm:act_gamma}, the element $\beta \coloneqq x^{-\chi_{1}(\gamma^{-1})}\gamma^{-1}(x)$ is contained in $\Pi(2,0)$. Since $f(\beta)\beta^{-1}$ is contained in $\Pi(\boldsymbol{m}+(2,0))$ by Lemma \ref{lmm:trivial_action}, we have
\begin{align*}
\gamma(f(\gamma^{-1}(x))\gamma^{-1}(x^{-1}))
&= \gamma(f(x^{\chi_{1}(\gamma^{-1})} \beta)\beta^{-1}x^{-\chi_{1}(\gamma^{-1})}) \\
&\equiv \gamma(f(x^{\chi_{1}(\gamma^{-1})})x^{-\chi_{1}(\gamma^{-1})}) \bmod \Pi(\boldsymbol{m}+(2,0)) \\
&= \chi_{1}^{m_{1}+1}(\gamma)(f(x^{\chi_{1}(\gamma^{-1})})x^{-\chi_{1}(\gamma^{-1})})
\end{align*}
Hence, to obtain the desired equality, it suffices to show  the congruence 
\[f(x^n)x^{-n} \equiv (f(x)x^{-1})^{n} \bmod \Pi(\boldsymbol{m}+(2,0)) \] for every $n \in \mathbb{Z}_{p}$. By continuity, it suffices to prove the congruence for every $n \in \mathbb{Z}_{\geq 1}$, which we prove by induction on $n$. We  compute $f(x^{n})x^{-n}$ as follows:
\begin{align*}
f(x^{n})x^{-n} 
&= f(x)f(x^{n-1})x^{-(n-1)}x^{-1} \\
&\equiv  f(x)(f(x)x^{-1})^{n-1}x^{-1} \bmod \Pi(\boldsymbol{m}+(2,0)) \\
&= f(x)(f(x)x^{-1})^{n-1}f(x)^{-1}(f(x)x^{-1}) = [f(x), (f(x)x^{-1})^{n-1}](f(x)x^{-1})^{n} \\
&\equiv (f(x)x^{-1})^{n} \bmod \Pi(\boldsymbol{m}+(2,0)).
\end{align*} Here, the induction hypothesis is used to deduce the second congruence, and we use 
\[
[f(x), (f(x)x^{-1})^{n-1}] \in \Pi(\boldsymbol{m}+(2,0))
\] to establish the last congruence. This concludes the proof.
\end{proof}

\begin{remark}
    Proposition \ref{prp:act_gamma_3} is an analogue of the fact that the map $\tilde{\iota}_{m}$ in \cite[(1.11.2)]{NT} is compatible with the action of the product $\mathrm{GSp}(2g,n) \coloneqq \mathrm{GSp}(2g) \times \mathfrak{S}_{n}$ of the symplectic similitude group over $\mathbb{Z}_{p}$ and the $n$-th symmetric group, which is proved in \cite[Theorem (1.11.4)]{NT}.
\end{remark}

By Proposition \ref{prp:act_gamma_3}, we now understand how $\Gamma^{\dag}_{1}/F^{(1,0)}\Gamma^{\dag}$ acts on $gr^{\boldsymbol{m}}_{1}\tilde{\Gamma}$. The following lemma allows us to describe the action of $\Gamma_{1}/F^{(1,0)}\Gamma$ on $gr^{\boldsymbol{m}}_{1}\Gamma$:

\begin{lemma}\label{(1,0)}
The natural homomorphisms
\[
\Gamma^{\dag}_{1}/F^{(1,0)}\Gamma^{\dag} \to \Gamma_{1}/F^{(1,0)}\Gamma
\quad \text{and} \quad
\Gamma^{\dag}_{2}/F^{(0,1)}\Gamma^{\dag} \to \Gamma_{2}/F^{(0,1)}\Gamma
\] are isomorphisms.
\end{lemma}

\begin{proof}
Since we have 
\[
\Ker(\Gamma^{\dag}_{1} \to  \Gamma_{1})=\Ker(F^{(1,0)}\Gamma^{\dag} \to F^{(1,0)}\Gamma)=\langle \mathrm{inn}(z) \rangle,
\] it suffices to prove that both $\Gamma^{\dag}_{1} \to \Gamma_{1}$ and $F^{(1,0)}\Gamma^{\dag} \to F^{(1,0)}\Gamma$ are surjective. First, we prove the surjectivity of $\Gamma^{\dag}_{1} \to \Gamma_{1}$. Let $\bar{\gamma} \in \Gamma_{1}$ be an arbitrary element and fix an arbitrary lift $\gamma \in \tilde{\Gamma}_{1}$ of $\bar{\gamma}$. If we write $\gamma(z)=gz^{\alpha}g^{-1}$ for some $g \in \Pi$ and $\alpha \in \mathbb{Z}_{p}^{\times}$, $g^{-1}\gamma g=\mathrm{inn}(g^{-1}) \circ \gamma$ preserves $\langle z \rangle$. Moreover, 
\[
(g^{-1} \gamma g)(y)y^{-1}=g^{-1}\gamma(y)gy^{-1}=[g^{-1}, \gamma(y)]\gamma(y)y^{-1} \in \Pi(2).
\] This shows $g^{-1}\gamma g \in \Gamma^{\dag}_{1}$, hence the surjectivity of $\Gamma^{\dag}_{1} \to \Gamma_{1}$. Regarding the surjectivity of $F^{(1,0)}\Gamma^{\dag} \to F^{(1,0)}\Gamma$, let $\bar{\gamma} \in F^{(1,0)}\Gamma$ be arbitrary element and fix a lift $\gamma \in F^{(1,0)}\tilde{\Gamma}$ of $\bar{\gamma}$. If we write $\gamma(z)=gz^{\alpha}g^{-1}$ for some $g \in \Pi$ and $\alpha \in \mathbb{Z}_{p}^{\times}$, then we have $g^{-1}\gamma g \in \Gamma^{\dag}_{1}$ by the above argument. Observe that
\[
\gamma(z)z^{-1}=[g, z^{\alpha}]z^{\alpha-1} \in \Pi(3)
\Rightarrow z^{\alpha-1} \in \Pi(3) \Rightarrow \alpha=1.
\]
Hence, to prove $(g^{-1}\gamma g)(x)x^{-1}=[g^{-1}, \gamma(x)]\gamma(x)x^{-1} \in \Pi(2,0)$, it is enough to show $g \in \Pi(1,0)$. Note that $\Pi(\boldsymbol{1})/\Pi(2,1)=\Pi(\boldsymbol{1})/\Pi(2,0) \subset \Pi(1,0)/\Pi(2,0)$ is a free $\mathbb{Z}_{p}[[\Pi/\Pi(1,0)]]$-submodule generated by $z=(y-1)x$. Hence that $\gamma(z)z^{-1}=[g,z]=(g-1)z$ is contained in $\Pi(2,0)$ implies $g=1$ in $\mathbb{Z}_{p}[[\Pi/\Pi(1,0)]]$, i.e. $g \in \Pi(1,0)$. This concludes the proof.
\end{proof}

To summarize the discussion so far, we obtain the following result on $\Gamma$:

\begin{proposition}\label{prp:act_gamma4}
    Let $\boldsymbol{m}=(m_{1},m_{2}) \in \mathbb{Z}_{\geq 0}^{2} \setminus \{ (0,0) \}$. The action of $\Gamma_{1}/F^{(1,0)}\Gamma$ (resp. $\Gamma_{2}/F^{(0,1)}\Gamma$) on $\Gr^{\boldsymbol{m}}_{1} \Gamma$ (resp. $\Gr^{\boldsymbol{m}}_{2} \Gamma$) is given by $\chi_{1}^{m_{1}}$ (resp. $\chi_{2}^{m_{2}}$).
\end{proposition}
\begin{proof}
    The assertion follows from Proposition \ref{prp:act_gamma_3} and Lemma \ref{(1,0)}.
\end{proof}

The following lemma will be used in the next subsection:

\begin{lemma}\label{lmm:kernel}
We have
\[
F^{(1,0)}\Gamma^{\dag}=
\Ker(\Gamma^{\dag}_{1} \xrightarrow{\chi_{1}} \mathbb{Z}_{p}^{\times})
\quad \text{and} \quad
F^{(0,1)}\Gamma^{\dag}=
\Ker(\Gamma^{\dag}_{2} \xrightarrow{\chi_{2}} \mathbb{Z}_{p}^{\times}).
\]
\end{lemma}

\begin{proof}
We only prove the first equality. Let $\gamma$ be an element of $\Ker(\Gamma^{\dag}_{1} \xrightarrow{\chi_{1}} \mathbb{Z}_{p}^{\times})$. Since we have $\gamma(y)y^{-1} \in \Pi(2)=\Pi(1,1)$, it suffices to show $\gamma(x)=x \bmod \Pi(2,0)$. Since the action of $\Gamma^{\dag}_{1}$ on $\Pi(1,0)/\Pi(2,0)$, which is a free $\mathbb{Z}_{p}[[\Pi/\Pi(1,0)]]$-module generated by the image of $x$,  is $\mathbb{Z}_{p}[[\Pi/\Pi(1,0)]]$-linear, the desired equality follows from the equality $\gamma(z)=z$. 
\end{proof}

\subsection{Two-variable filtration on Galois group}

In the following, we apply various results obtained in this section to the case of the pro-$p$ fundamental group $\Pi=\Pi_{1,1}$. We use our fixed basis $x=x_{1}$ and $y=x_{2}$ of $\Pi_{1,1}$. Recall that we have the pro-$p$ outer Galois representation
\[
\rho_{X, p} \colon G_{K} \to \Out(\Pi_{1,1}),
\] and the image of $\rho_{X,p}$ is contained in $\Gamma \subset \Gamma_{1,1}$ introduced in the last subsection. For $\boldsymbol{m} \in \mathbb{Z}_{\geq 0}^{2} \setminus \{ (0,0) \}$, let $F^{\boldsymbol{m}}G_{K} \subset F^{\lvert \boldsymbol{m} \rvert }G_{K}$ denote the inverse image of $F^{\boldsymbol{m}}\Gamma$ under $\rho_{X,p}$. We define the associated graded quotients by
\[
\Gr^{\boldsymbol{m}}_{1}G_{K} \coloneqq F^{\boldsymbol{m}}G_{K}/F^{\boldsymbol{m}+(1,0)}G_{K}
\quad \text{and} \quad \Gr^{\boldsymbol{m}}_{2}G_{K} \coloneqq F^{\boldsymbol{m}}G_{K}/F^{\boldsymbol{m}+(0,1)}G_{K}.
\] Moreover, let $F_{1}G_{K}$ (resp. $F_{2}G_{K}$) be the inverse image of $\Gamma_{1}$ (resp. $\Gamma_{2}$). 

\begin{lemma}\label{lmm:Gal_1}
We have the following equalities:
\begin{enumerate}
\item[(1)] $F_{1}G_{K}=G_{K(E[\bar{\p}^{\infty}])}$ and $F_{2}G_{K}=G_{K(E[\p^{\infty}])}$.
\item[(2)] $F^{(1,0)}G_{K}=F^{(0,1)}G_{K}=F^{(1,1)}G_{K}=G_{K(E[p^{\infty}])}$.
\end{enumerate}
\end{lemma}
\begin{proof}
The assertion of (1) is clear. By Lemma \ref{(1,0)} and Lemma \ref{lmm:kernel}, the group $F^{(1,0)}G_{K}$ coincides with the kernel of $F_{1}G_{K} \xrightarrow{\chi_{1}} \mathbb{Z}_{p}^{\times}$, and thus it coincides with $G_{K(E[p^{\infty}])}$. The same argument shows $F^{(0,1)}G_{K}=G_{K(E[p^{\infty}])}$. Since we have
\[
F^{(1,0)}G_{K} \cap F^{(0,1)}G_{K}=F^{(1,1)}G_{K},
\] the last assertion also follows.
\end{proof}

The following corollary is one of key ingredients to prove Theorem \ref{thm:main}:

\begin{corollary}\label{cor:Gal_2}
Let $\boldsymbol{m} \in \mathbb{Z}_{\geq 0}^{2} \setminus \{ (0,0) \}$. The action of $\Gal(K(E[p^{\infty}])/K(E[\bar{\p}^{\infty}]))$ on $\Gr^{\boldsymbol{m}}_{1}G_{K}$ is given by the character $\chi_{1}^{m_{1}}$. Similarly, the action of $\Gal(K(E[p^{\infty}])/K(E[\p^{\infty}]))$ on $\Gr^{\boldsymbol{m}}_{2}G_{K}$ is given by $\chi_{2}^{m_{2}}$.
\end{corollary}

\begin{proof} 
Both assertions immediately follow from Proposition \ref{prp:act_gamma4} and Lemma \ref{lmm:Gal_1}.
\end{proof}

\section{Proof of main result}\label{4}

In  this section, we keep the same notation as in Section \ref{2}. Moreover, we abbreviate $\Gal(\Omega/K(p))$ and $\Gal(\Omega^{\ast}/K(p))$ as $G$ and $G^{\ast}$, respectively. Recall that the index set $I$ is defined by
\[
I = \{ \boldsymbol{m}=(m_{1}, m_{2}) \in \mathbb{Z}_{ \geq 1}^{2} \setminus \{ \boldsymbol{1} \} \mid m_{1} \equiv m_{2} \bmod \lvert O_{K}^{\times} \rvert \}.
\] In this section, we also use the following subset:
\[
I_{0} \coloneqq \{ \boldsymbol{m} \in I \mid (p-1, p-1) \geq \boldsymbol{m} \} \cup \{(p,1), (1,p) \}.
\]

We first give two filtrations on $G^{\ast}$.  By Lemma \ref{lmm:half}, we have 
\[
\rho_{X,p}(G_{K(p)})=G^{\ast} \times \Gal(K(E[p])/K(p)) \subset \Gamma_{1,1}.
\] The image $\rho_{X,p}(G_{K(p)})$ has a descending central filtration defined at the beginning of Section \ref{2.1}. Moreover, it is also equipped with a two-variable filtration induced from the two-variable filtration on $\Gamma$ as in the previous subsection. By taking the images under the projection $\rho_{X,p}(G_{K(p)}) \to G^{\ast}$, the group $G^{\ast}$ is also equipped with a descending central filtration $\{ F^{m}G^{\ast} \}_{m \geq 1}$ and the two-variable filtration $\{ F^{\boldsymbol{m}}G^{\ast} \}_{\boldsymbol{m}}$. 

Moreover, by taking pullbacks of these filtrations under the natural projection, we obtain filtrations $\{F^{m}G\}_{m \geq 1}$ and $\{ F^{\boldsymbol{m}}G \}_{\boldsymbol{m}}$. Then, by construction, the graded Lie algebras associated to $\{F^{m}G^{\ast} \}_{m \geq 1}$ and $\{ F^{m}G  \}_{m \geq 1}$ are isomorphic to $\mathfrak{g}=\oplus_{m \geq 1}\mathfrak{g}_{m}$ defined in Section \ref{2.1}.

In the subsequent subsections, we give a proof of Theorem \ref{thm:main}. The proof follows Sharifi's approach \cite[Theorem 1.2]{Sh}. Namely, we show that the intersection $\cap_{m \geq 1} F^{m}G=\Gal(\Omega/\Omega^{\ast})$ is trivial by the following strategy.

\begin{enumerate}
\item[(1)] In Section \ref{4.1}, we construct an element $\sigma_{\boldsymbol{m}} \in F^{\boldsymbol{m}}G$ whose image in $\mathfrak{g}_{|\boldsymbol{m}|}$ satisfies the assumption of Conjecture \ref{cnj:DI11}.  Moreover, it is proved in Section \ref{4.2} that $\{ \sigma_{\boldsymbol{m}} \}_{\boldsymbol{m} \in I}$ strongly generates $F^{1}G$. Here, several properties of two-variable filtrations are used in an essential way.
\item[(2)] Conjecture \ref{cnj:DI11} implies that $\{ F^{m}G \}_{m \geq 1}$ coincides with the ``fastest'' descending central filtration $\{ \tilde{F}^{m}G \}_{m \geq 1}$ satisfying $\sigma_{\boldsymbol{m}} \in  \tilde{F}^{|\boldsymbol{m}|}G$ for every $\boldsymbol{m} \in I$. Since the intersection of the latter filtration is proved to be trivial by Lemma \ref{lmm:filt} (3) in a purely group-theoretic way, we obtain the desired result.
\end{enumerate}

We first prepare a lemma. This allows us to apply Proposition \ref{prp:DI11_Soule} to construct a basis of $\mathfrak{g} \otimes \mathbb{Q}_{p}$ under the assumption of Theorem \ref{thm:main}: 

\begin{lemma}\label{lmm:nontrivial}
    Assume that the class number of $K(p)$ is not divisible by $p$ and there are exactly two primes of $K(p^{2})$ above $p$. Then the character $\kappa_{\boldsymbol{m}}$ is nontrivial for every $\boldsymbol{m} \in I$.
\end{lemma}
\begin{proof}
    The assertion follows from Theorem \ref{thm:2} and Corollary \ref{cor:app1}.
\end{proof}

\subsection{Construction of elements}\label{4.1}

In this subsection, we construct elements $\{ \sigma_{\boldsymbol{m}} \}_{\boldsymbol{m} \in I}$ satisfying the assumption of Conjecture \ref{cnj:DI11} in such a way that can be regarded as a two-variable variant of Sharifi's construction in \cite[2]{Sh}. We also construct auxiliary elements $\{ g_{\boldsymbol{m}} \}_{\boldsymbol{m} \in I}$, which are similar to $\{ \sigma_{\boldsymbol{m}} \}_{\boldsymbol{m} \in I}$ but much easier to handle. These auxiliary elements play an important role when proving that $\{ \sigma_{\boldsymbol{m}} \}_{\boldsymbol{m} \in I}$ strongly generate $F^{1}G$, see the proof of Theorem \ref{thm:main} in Section \ref{4.2} for more details.

First, we lift generators of $\Gal(K(p^{\infty})/K) \cong \Delta \times \mathbb{Z}_{p}^{2}$ to $\Gal(\Omega/K)$, as follows: Let us denote the maximal pro-$p$ subextension of $K(p^{\infty})/K$ by $K_{\infty}/K$. Then it is a $\mathbb{Z}_{p}^{2}$-extension. The upper exact sequence in the following commutative diagram
\[
\begin{tikzcd}
1 \arrow[r] &  \Gal(\Omega/K(p^{\infty})) \arrow[r]\arrow[d, hook] & \Gal(\Omega/K_{\infty}) \arrow[r]\arrow[d, hook] & \Gal(K(p^{\infty})/K_{\infty}) \arrow[r]\arrow[d, "\rotatebox{90}{\(\sim\)}"] & 1 \\
1 \arrow[r] & G \arrow[r] & \Gal(\Omega/K) \arrow[r] & \Delta=\Gal(K(p)/K) \arrow[r] & 1 
\end{tikzcd}
\]
splits since $\Gal(\Omega/K(p^{\infty}))$ is a pro-$p$ group and $\Delta$ is a prime-to-$p$ group by Schur-Zassenhaus \cite[5.9, Corollary 1]{Se}. We fix a section $r \colon \Gal(K(p)/K) \to \Gal(\Omega/K_{\infty})$ and identify  $\Gal(\Omega/K)$ with the semi-direct product $G \rtimes \Delta$. Then $\Delta$ acts on $\Gal(\Omega/K(p^{\infty}))$ through this section by conjugation, and  
\[
1 \to  \Gal(\Omega/K(p^{\infty})) \to G \to \Gal(K(p^{\infty})/K(p)) \to 1
\] is an exact sequence of pro-$p$ groups with $\Delta$-action, noting that the action of $\Delta$ on $\Gal(K(p^{\infty})/K(p))$ is trivial. Then the sequence
\[
1 \to  \Gal(\Omega/K(p^{\infty}))^{\Delta} \to G^{\Delta} \to \Gal(K(p^{\infty})/K(p)) \to 1
\] obtained by taking $\Delta$-invariant subgroups is exact, since the homomorphism 
\[
G^{\Delta} \to \Gal(K(p^{\infty})/K(p))
\] is surjective by applying Lemma \ref{lmm:idemp} below to the case of $m=0$ below.

Let $\gamma_{1}$ (resp. $\gamma_{2}$) be an element of the $\Delta$-invariant subgroup $G^{\Delta}$ which restricts to a generator of $\Gal(K(p^{\infty})/K(\bar{\p}^{\infty}\p))\cong \mathbb{Z}_{p}$ (resp. $\Gal(K(p^{\infty})/K(\p^{\infty}\bar{\p})) \cong \mathbb{Z}_{p}$). Moreover, we fix a generator $\delta \in \mathbb{F}_{p}^{\times}$ and let $\delta_{1}$ (resp. $\delta_{2}$) be an element of $\Gal(\Omega/K)$ defined to be the image of the image of $(\delta,1)$ (resp. $(1, \delta)$) in $\Delta$ through the homomorphism \[
 (\mathbb{F}_{p}^{\times} \times \mathbb{F}_{p}^{\times})/\Img(O_{K}^{\times}) \xrightarrow{(\chi_{1},\chi_{2})^{-1}}
\Delta  \xrightarrow{r} \Gal(\Omega/K),
\] where the first arrow is an isomorphism induced by characters $\chi_{1}$ and $\chi_{2}$. Then $\delta_{1}$ (resp. $\delta_{2}$) is contained in $\Gal(\Omega/K(\bar{\p}))$ (resp. $\Gal(\Omega/K(\bar{p}))$) and, by construction, we have the following relations:
\[
[\delta_{1}, \delta_{2}]=1, \quad
[\delta_{1}, \gamma_{1}]=1, \quad
[\delta_{1}, \gamma_{2}]=1, \quad
[\delta_{2}, \gamma_{1}]=1 \quad \text{and} \quad
[\delta_{2}, \gamma_{2}]=1. \quad
\]

For $m \in \mathbb{Z}_{\geq 0}$ and $1 \leq i \leq 2$, we define an idempotent $\epsilon_{i, m} \in \mathbb{Z}_{p}[\Gal(\Omega/K)]$ by
\[
\epsilon_{i,m} \coloneqq \frac{1}{p-1}\sum_{j=0}^{p-2} \chi_{i}^{-m}(\delta_{i}^{j}) \delta_{i}^{j}.
\]

Moreover, for $g \in G$, we define 
\[
g^{\epsilon_{i,m}} \coloneqq (g \cdot \delta_{i}g^{\chi_{i}^{-m}(\delta_{i})}\delta_{i}^{-1} \cdots \delta_{i}^{p-2} g^{\chi_{i}^{-m}(\delta_{i}^{p-2})} \delta_{i}^{-(p-2)})^{\frac{1}{p-1}},
\] and let $g^{\epsilon_{i,m}^{j}}=(\cdots(g^{\epsilon_{i,m}})\cdots)^{\epsilon_{i,m}}$ denote its $j$-th iterate for $j \geq 1$. Since $G$ is nonabelian, we do not necessarily have the equality $g^{\epsilon_{i,m}^{j}}=g^{\epsilon_{i,m}}$. However, we have the following lemma, originally due to Sharifi, which is one of the key ingredients to construct desired elements $\{\sigma_{\boldsymbol{m}} \}_{\boldsymbol{m} \in I}$ group-theoretically. 

\begin{lemma}[cf. Sharifi {\cite[Lemma 2.1]{Sh}}]\label{lmm:idemp}
For every $g \in G$,  $m \in \mathbb{Z}_{\geq 0}$ and $1 \leq i \leq 2$, the limit
\[
g^{(i,m)} \coloneqq \lim_{j \to \infty} g^{\epsilon_{i,m}^{j}}
\] exists and satisfies $
\delta_{i}g^{(i,m)}\delta_{i}^{-1} \coloneqq (g^{(i,m)})^{\chi_{i}^{m}(\delta_{i})}.
$
\end{lemma}
\begin{proof}
    The assertion follows from the same argument as \cite[Lemma 2.1]{Sh}. 
\end{proof}

Let $A$ denote the maximal abelian quotient of $F^{1}G=\Gal(\Omega/K(p^{\infty}))$, which is naturally endowed with the structure of a $\Lambda \coloneqq \mathbb{Z}_{p}[[\Gal(K(p^{\infty})/K(p))]]$-module. In the following lemma and remark, we determine the structure of this $\Lambda$-module completely:

\begin{lemma}\label{Wintenberger}
    Suppose the class number of $K(p)$ is not divisible by $p$ and there are exactly two primes of $K(p^{2})$ lying above $p$. Let $\boldsymbol{m} \in I_{0}$ and let $A^{\boldsymbol{m}} \coloneqq \epsilon_{1,m_{1}}\epsilon_{2,m_{2}}A$. Then the following assertions hold:
\begin{enumerate}
\item[(1)] If $\boldsymbol{m} \neq (1,p)$ or $(p,1)$, then $A^{\boldsymbol{m}}$ is a cyclic $\Lambda$-module.
\item[(2)] $A^{(p,1)}=A^{(1,p)}$ is a quotient of the annihilator ideal $\Ann_{\Lambda}(\mathbb{Z}_{p}(1))$ of $\mathbb{Z}_{p}(1)$.
\end{enumerate}
\end{lemma}
\begin{proof}
We have the five-term exact sequence between $\mathbb{F}_{p}[\Delta]$-modules \cite[(1.6.7)]{NSW}
\begin{align*} 
0 &\to H^{1}(\Gal(K(p^{\infty})/K(p)), \mathbb{F}_{p}) \to  H^{1}(G , \mathbb{F}_{p}) \to
\Hom_{\Gal(K(p^{\infty})/K(p))}(A, \mathbb{F}_{p}) 
\\
&\to H^{2}(\Gal(K(p^{\infty})/K(p)), \mathbb{F}_{p}) \to H^{2}(G, \mathbb{F}_{p}). 
\end{align*} Note that
$
\Hom_{\Gal(K(p^{\infty})/K(p))}(A, \mathbb{F}_{p}) 
=\oplus_{\boldsymbol{m} \in I_{0} \setminus \{ (p,1) \}}\Hom_{\Gal(K(p^{\infty})/K(p))}(A^{\boldsymbol{m}}, \mathbb{F}_{p}).
$
 To prove the assertion of (1), it suffices to compute the dimension of each eigenspace of this cohomology group. First, we compute the dimension of each eigenspace of $H^{1}(G, \mathbb{F}_{p})$. The Kummer exact sequence $1 \to \mu_{p} \to \mathbb{G}_{m} \to \mathbb{G}_{m} \to 1$ gives
\[
1 \to O_{K(p)}[1/p]^{\times}/p \to H^{1}(G, \mu_{p}) \to H^{1}_{\et}(O_{K(p)}[1/p], \mathbb{G}_{m})[p] \to 0.
\] Since the class number of $K(p)$ is prime to $p$, the group $H^{1}_{\et}(O_{K(p)}[1/p], \mathbb{G}_{m})[p]$, which is nothing but the $p$-torsion subgroup of the $p$-class group, is trivial. Hence we obtain an isomorphism
\[
    O_{K(p)}[1/p]^{\times}/p \xrightarrow{\sim} H^{1}(G, \mu_{p}).
\] Since there are only two primes of $K(p)$ lying above $p$ by assumption, the dimension of the left hand side is equal to $[K(p):K]+2$ by Dirichlet's unit theorem. Moreover, by \cite[(8.7.2) Proposition]{NSW}, there is an isomorphism of $\mathbb{Q}[\Delta]$-modules
\[
O_{K(p)}[1/p]^{\times} \otimes \mathbb{Q} \cong \mathbb{Q}[\Delta]  \oplus \mathbb{Q}.
\] 
By decomposing the $p$-unit group into the product of the torsion-part and the free-part, it follows that the dimension of $\chi^{\boldsymbol{m}}$-component of $H^{1}(G, \mathbb{F}_{p})$ is at most one if $\boldsymbol{m} \in I_{0} \setminus \{ (1,p), (p,1), (p-1,p-1) \}$ and otherwise at most two. By a counting argument, all these inequalities are equalities. Since $\Delta$ acts trivially on $H^{i}(\Gal(K(p^{\infty})/K(p)), \mathbb{F}_{p})$ ($i=1,2$), (1) follows from the five-term exact sequence.

By our assumption, there exists a unique prime of $K(p^{\infty})$ above $\p$ which we denote by the same letter $\p$ (see Remark \ref{rmk:secondcond}). Hence we have an isomorphism
\[
\Gal(K(p^{\infty})_{\p}/K(p)_{\p}) \xrightarrow{\sim} \Gal(K(p^{\infty})/K(p)).
\]
Moreover, by a result of Wintenberger \cite[TH\'EOR\`EME]{Win}, the decomposition group of $A^{(1,p)}$ at $\p$ is isomorphic to
\[
\Ann_{\mathbb{Z}_{p}[[\Gal(K(p^{\infty})_{\p}/K(p)_{\p})]]}(\mathbb{Z}_{p}(1))=\Ann_{\Lambda}(\mathbb{Z}_{p}(1)).
\] Hence to prove (2), it suffices to show that $A^{(1,p)}$ coincides with its decomposition subgroup at $\p$. First, note that 
\[
 \Hom_{\Delta}(A^{(1,p)}, \mathbb{F}_{p}(1)) \cong \Hom_{\Delta}(\Gal(\Omega/K(p)), \mathbb{F}_{p}(1)) \cong O_{K}[1/p]^{\times}/p \cong \mathbb{F}_{p}^{2}
\] is generated by $\pi$ and $\bar{\pi}$. Hence the assertion is equivalent to saying that the images of $\pi$ and $\bar{\pi}$ in $K_{\p}^{\times}/(K_{\p}^{\times})^{p}$ still span a two-dimensional subspace. This assertion is then equivalent to saying that the projection of $\bar{\pi} \in O_{K_{\p}}^{\times}$ to the group of principal units $1+\p O_{K_{\p}} \cong \mathbb{Z}_{p}$ is a generator. However, this follows since the Frobenius element at $\bar{\p}$ in $\Gal(K(\p)/K)$, which coincides with $\bar{\pi}$ under $\Gal(K(\p)/K) \cong (O_{K_{\p}}/\p)^{\times}/O_{K}^{\times}$, is a generator by our assumption on the number of primes above $p$.
\end{proof}

In fact, we can determine the $\Lambda$-module structure of $A$ and the structure of the pro-$p$ group $\Gal(\Omega/K(p^{\infty}))$ completely, though they are not necessary to prove Theorem \ref{thm:main}. Here we record the precise statement and provide its proof for the interested reader:

\begin{proposition}\label{prp:Iwasawa} 
     Suppose the class number of $K(p)$ is not divisible by $p$ and there are exactly two primes of $K(p^{2})$ above $p$. The following assertions hold.
     \begin{enumerate}
        \item There exists an isomorphism
        \[
         \Lambda^{[K(p):K]-1} \oplus \Ann_{\Lambda}(\mathbb{Z}_{p}(1)) \xrightarrow{\sim} A
        \] between $\Lambda$-modules.
        \item The group  $\Gal(\Omega/K(p^{\infty}))$ is a free pro-$p$ group (of countably infinite rank).
     \end{enumerate}
\end{proposition}
\begin{proof}
    (1) The proof of Lemma \ref{Wintenberger} shows that there exists a surjective homomorphism 
    \[
    f \colon \Lambda^{[K(p):K]-1} \oplus \Ann_{\Lambda}(\mathbb{Z}_{p}(1)) \to A
    \] between $\Lambda$-modules. Moreover, by \cite[Corollaire 2.7]{NQD}, we have
    \[
    \dim_{\mathrm{Frac}(\Omega)} A \otimes_{\Lambda} \mathrm{Frac}(\Omega)=[K(p):K]. 
    \] This forces the kernel of $f$ to be a torsion $\Lambda$-module. However, since the left-hand side is obviously torsion-free, the assertion follows.

    (2) The assertion (1) implies that $A=\Gal(\Omega/K(p^{\infty}))^{\ab}$ is $\mathbb{Z}_{p}$-torsion free. Moreover, we have $H^{2}(\Gal(\Omega/K(p^{\infty})), \mathbb{Q}_{p}/\mathbb{Z}_{p})=0$, i.e. the weak Leopoldt conjecture holds for the $\mathbb{Z}_{p}^{2}$-extension $K(p^{\infty})/K(p)$ by \cite[Theoreme 2.2]{NQD}, since it contains the cyclotomic $\mathbb{Z}_{p}$-extension. Now, by considering a long exact sequence associated to 
    \[
    0 \to \mathbb{F}_{p} \to \mathbb{Q}_{p}/\mathbb{Z}_{p} \xrightarrow{ p}  \mathbb{Q}_{p}/\mathbb{Z}_{p} \to 0,
    \] we obtain $H^{2}(\Gal(\Omega/K(p^{\infty})), \mathbb{F}_{p})=0$. This concludes the proof.
\end{proof}

Now we construct elements $\sigma_{\boldsymbol{m}} \in F^{\boldsymbol{m}}G_{K}$ for $\boldsymbol{m} \in I_{0}$. In the rest of this section, since our construction relies on Lemma \ref{Wintenberger}, we suppose $p$ satisfies the following two assumptions of Theorem \ref{thm:main}:

\medskip

\fbox{
    \begin{minipage}{0.7\linewidth}
        \begin{enumerate}
            \item  The class number of $K(p)$ is not divisible by $p$.
            \item  There are exactly two primes of $K(p^{2})$ above $p$. 
        \end{enumerate}
    \end{minipage}
}

\medskip

{\bf Construction. }   For $\boldsymbol{m} \in I_{0}$, we choose an element $t_{\boldsymbol{m}} \in \Gal(\Omega/K(p^{\infty}))$, as follows:

\begin{itemize}
\item If $\boldsymbol{m} \in I_{0} \setminus \{(p,1), (1,p), (p-1,p-1)\}$, we choose a lift $t_{\boldsymbol{m}} \in \Gal(\Omega/K(p^{\infty}))$ of a generator of $A^{\boldsymbol{m}}$ as $\mathbb{Z}_{p}[[\Gal(K(p^{\infty})/K(p))]]$-module. 

\item For $\boldsymbol{m}=(p, 1)$ and $(1,p)$, fix a surjection from $M$ given above to $A^{(p,1)}$ and let $t_{(p,1)}$ and $t_{(1,p)}$ be arbitrary lifts of the images of $v_{1}$ and $v_{2}$, respectively.

\item For $\boldsymbol{m}=(p-1,p-1)$, set $t_{(p-1,p-1)}=[\gamma_{1}, \gamma_{2}]$.
\end{itemize}

For every $\boldsymbol{m} \in I_{0}$, let
\[
\sigma_{\boldsymbol{m}} \coloneqq \left( t_{\boldsymbol{m}}^{(1,m_{1})} \right)^{(2,m_{2})} \quad \text{and} \quad g_{\boldsymbol{m}} \coloneqq \sigma_{\boldsymbol{m}}.
\] 

Note that $\sigma_{(p-1,p-1)}=g_{(p-1,p-1)}=[\gamma_{1}, \gamma_{2}]$ since $\gamma_{1}$ and $\gamma_{2}$ commute with $\delta_{1}$ and $\delta_{2}$. The elements $\{ \sigma_{\boldsymbol{m}} \}_{\boldsymbol{m} \in I_{0}}$ satisfy the following properties:

\begin{lemma}\label{lmm:lift_1}
 For $\boldsymbol{m} \in I_{0}$, the following two assertions hold.
\begin{enumerate}
\item[(1)] The element $\sigma_{\boldsymbol{m}}$ is contained in $F^{\boldsymbol{m}}G$, and its image in $\mathfrak{g}_{\lvert \boldsymbol{m} \rvert}$ is contained in the $\chi^{\boldsymbol{m}}$-isotypic component.
\item[(2)] The element $\kappa_{\boldsymbol{m}}(\sigma_{\boldsymbol{m}})$ generates $\kappa_{\boldsymbol{m}}(F^{1}G)$, which is nonzero.
\end{enumerate}
\end{lemma}

Before proving the lemma, we first prove the following two lemmas concerning the case where $m_{1}=1$ or $m_{2}=1$.

\begin{lemma}\label{lmm:(1,p)(p,1)}
    Let  $n \geq 2$ be an integer such that $n \equiv 1 \bmod p-1$. Then we have $\kappa_{(1,n)}(t_{(p,1)})=0$. Similarly, for such $n$, we have $\kappa_{(n,1)}(t_{(1,p)})=0$.
\end{lemma}
\begin{proof}
Note that $\kappa_{(n,1)}$ factors through $A^{(1,1)}$, and we have 
\[
(\gamma_{2}-\chi_{2}(\gamma_{2}))t_{(p,1)}=(\gamma_{1}-\chi_{1}(\gamma_{1}))t_{(1,p)}
\] in $A^{(1,1)}$ by definitions of $t_{(p,1)}$ and $t_{(1,p)}$. Hence we have
\begin{align*}
\kappa_{(1,n)}((\gamma_{2}-\chi_{2}(\gamma_{2}))t_{(p,1)})
&= (\chi_{2}^{n}(\gamma_{2})-\chi_{2}(\gamma_{2}))\kappa_{(1,n)}(t_{(p,1)}) \\
&= \kappa_{(1,n)}((\gamma_{1}-\chi_{1}(\gamma_{1}))t_{(1,p)}) \\
&= (\chi_{1}(\gamma_{1})-\chi_{1}(\gamma_{1}))\kappa_{p,1}(t_{(1,p)})=0.
\end{align*} Since $\chi_{2}^{n-1}(\gamma_{2}) \neq 1$, it follows that $\kappa_{(1,n)}(t_{(p,1)})=0$. The same argument shows $\kappa_{(n,1)}(t_{(1,p)})=0$, as desired.
\end{proof}

\begin{lemma}\label{lmm:(1,2)(2,1)}
We have $t_{(p,1)} \in F^{(2,1)}G$ and $t_{(1,p)} \in F^{(1,2)}G$.
\end{lemma}
\begin{proof}
    
Let $s \colon F^{1}G_{K} \to F^{3}\Gamma^{\dag}_{1,1}$ be the lift of $\rho_{X,p}$ constructed in Section \ref{2.2}. Then the map $s$ factors through $F^{1}G_{K} \to F^{1}G$, and we denote the resulting homomorphism $F^{1}G \to F^{3}\Gamma^{\dag}_{1,1}$ by the same letter. By Lemma \ref{lmm:contain}, to show $t_{(p,1)} \in F^{(2,1)}G$, it suffices to prove that $s(t_{(p,1)})(x_{2})x_{2}^{-1} \in \Pi(2,2)$ and $s(t_{(p,1)})(x_{1})x_{1}^{-1} \in \Pi(3,1)$.

First, we prove that $s(t_{(p,1)})(x_{2})x_{2}^{-1} \in \Pi(2,2)$. To verify this inclusion, it is enough to show that the power series $\alpha_{1,1}(t_{(p,1)})$ is divisible by $T_{1}$. In fact, then we would have 
\begin{align*}
    s(t_{(p,1)})(x_{2})x_{2}^{-1}
    &= \alpha_{1,1}(t_{(p,1)})[x_{2},z] \\
    &= \frac{\alpha_{1,1}(t_{(p,1)})}{T_{1}}[x_{1}, [x_{2},z]] \in \Pi(2)/[\Pi(2), \Pi(2)].
\end{align*} by definition of $\alpha_{1,1}$. Since it holds that 
\[
[\Pi(2), \Pi(2)]=[\Pi(1,1), \Pi(1,1)] \subset \Pi(2,2)
\] and $[x_{1},[x_{2},z]] \in \Pi(2,2)$, we have $s(t_{(p,1)})(x_{2})x_{2}^{-1} \in \Pi(2,2)$ as desired.

In view of the explicit formula for $\alpha_{1,1}$ (Theorem \ref{thm:Na}), it suffices to show that 
\[
\kappa_{\boldsymbol{n}}(t_{(p,1)})=0 \quad \text{for every }  \boldsymbol{n}=(1, n_{2}) \in I.
\] For $\boldsymbol{n}=(1,n_{2}) \in I$, we have $\kappa_{\boldsymbol{n}}(t_{(p,1)})=0$ unless $n_{2} \equiv 1 \bmod p-1$, since $\kappa_{(1,n)}$ factors through $A^{(1,n)}$ while $t_{(p,1)}$ is a lift of an element of $A^{(p,1)}=A^{(1,p)}$. Moreover, by Lemma \ref{lmm:(1,p)(p,1)} above, we also have $\kappa_{\boldsymbol{n}}(t_{(p,1)})=0$ when $n_{2} \equiv 1 \bmod p-1$. This implies that $\alpha_{1,1}(t_{(p,1)})$ is divisible by $T_{1}$ as desired.

In the following, we prove $s(t_{(p,1)})(x_{1})x_{1}^{-1}  \in \Pi(3,1)=\Pi(3,0)$ (here, we use Lemma \ref{lmm:incl}). We have already shown that this element is contained in $\Pi(2,0)$ and, from the equality $s(t_{(p,1)})(z)=z$ and $s(t_{(p,1)})(x_{2})x_{2}^{-1} \in \Pi(2,2)$, it follows that \[[s(t_{(p,1)})(x_{1})x_{1}^{-1}, x_{2}] \in \Pi(3,2) \subset \Pi(3,0).\] In particular, we have the following:

\begin{center}
$s(t_{(p,1)})(x_{1})x_{1}^{-1} \in \Pi(2,0)/\Pi(3,0)$ is invariant under the conjugation of $x_{2}$.
\end{center}

Hence it suffices to show the following claim: 

\medskip

{\bf Claim.} The identity is the only element of $\Pi(2,0)/\Pi(3,0)$ which is invariant under the conjugation of $x_{2}$. 

\medskip

To prove the claim, first note that, for every $m \geq 1$, the group $\Pi(m,0)$ is just the $m$-th component of the descending central series of $\Pi(1,0)$, which is a free pro-$p$ group on the set $\{ w_{n} \}_{n \geq 1}$ where $w_{0} \coloneqq x_{1}$ and $w_{n} \coloneqq [x_{2}, w_{n-1}]$ for every $n \geq 1$. 

For $n \geq 1$, let $F_{n}$ be the quotient of $\Pi(1,0)$ by the normal closure of $\{ w_{i} \}_{i \geq n}$. Then $F_{n}$ is a free pro-$p$ group on the set $\{ w_{i} \}_{0 \leq i < n}$, and $\Pi(1,0)$ is isomorphic to $\varprojlim_{n} F_{n}$. Note that the commutator map induces an isomorphism
\[
F_{n}/F_{n}(2) \wedge F_{n}/F_{n}(2) \xrightarrow{\sim} F_{n}(2)/F_{n}(3)
\] for every $n \geq 1$, since the Lie algebra associated to the descending central series of $F_{n}$ is freely generated by the image of $\{ w_{i} \}_{0 \leq i < n}$ in $F_{n}/F_{n}(2)$. In other words, the quotient $F_{n}(2)/F_{n}(3)$ is a free $\mathbb{Z}_{p}$-module with basis $\{ [w_{i}, w_{j}] \}_{0 \leq i < j < n}$. Therefore, we have
\[
 \Pi(2,0)/\Pi(3,0)=\varprojlim_{n} F_{n}(2)/F_{n}(3)=\prod_{0 \leq i < j} \mathbb{Z}_{p}[w_{i}, w_{j}].
\] Observe that the action of $x_{2}$ sends $[w_{i}, w_{j}] \in \Pi(2,0)/\Pi(3,0)$ to 
\[
[w_{i+1}w_{i}, w_{j+1}w_{j}]=[w_{i+1}, w_{j+1}]+[w_{i+1}, w_{j}]+[w_{i}, w_{j+1}]+[w_{i}, w_{j}].
\] If $v=(v_{i,j}) \in \prod_{0 \leq i < j} \mathbb{Z}_{p}[w_{i}, w_{j}]=\Pi(2,0)/\Pi(3,0)$ is invariant under the conjugation of $x_{2}$, then one can show that $v_{0,j}=0$ for every $j>0$ by induction on $j$. By repeating induction for every $i>0$, it follows that $v_{i,j}=0$ for every $0 \leq i < j$, hence $v=0$ as desired.
\end{proof}

\begin{proof}[Proof of Lemma \ref{lmm:lift_1}]
First, note that $\sigma_{\boldsymbol{m}} \in F^{1}G=F^{\boldsymbol{1}}G$ by Lemma \ref{lmm:Gal_1}. Hence the assertion of (1) immediately follows from Corollary \ref{cor:Gal_2}, except when $\boldsymbol{m}=(1,p)$ or $(p,1)$. If $\boldsymbol{m}=(p,1)$, we know $t_{(p,1)} \in F^{(2,1)}G$ by Lemma \ref{lmm:(1,2)(2,1)}. Then the claim that $\sigma_{(p,1)} \in F^{(p,1)}G$ follows from Corollary \ref{cor:Gal_2}, and the case where $\boldsymbol{m}=(1,p)$ is similar. The second assertion of (1) also follows from Corollary \ref{cor:Gal_2}.
 
Next we prove the assertion of (2). For $\boldsymbol{m} \in I_{0} \setminus \{(1,p), (p,1), (p-1,p-1)\}$, the assertion immediately follows since we have $\kappa_{\boldsymbol{m}}(t_{\boldsymbol{m}})=\kappa_{\boldsymbol{m}}(\sigma_{\boldsymbol{m}})$ and the element $t_{\boldsymbol{m}}$ generates $A^{\boldsymbol{m}}$ as a  $\Lambda$-module. Since the character $\kappa_{\boldsymbol{m}}$ is nontrivial by Lemma \ref{lmm:nontrivial}, the image $\kappa_{\boldsymbol{m}}(F^{1}G)$ is nonzero. 

Assume that $\boldsymbol{m}=(p-1,p-1)$. Since $A^{(p-1,p-1)}=A^{\Delta}$, the $\Delta$-invariant part of $A$, it follows that $A^{(p-1,p-1)}$ is isomorphic to $\Gal(L/K_{\infty})^{{\rm ab}}$, where $L$ is the maximal pro-$p$ extension of $K$ unramified outside $p$.  Since $K$ has class number one, it follows by \cite[(10.7.13) Theorem]{NSW} that $\Gal(L/K)$ is a free pro-$p$ group of rank two on the set $\{ \gamma_{1}, \gamma_{2} \}$. Hence $\Gal(L/K_{\infty})^{\rm{ab}}$, which is just the maximal abelian quotient of the commutator subgroup of a free pro-$p$ group of rank two, is a free $\Lambda$-module of rank one generated by $[\gamma_{1}, \gamma_{2}]$ by \cite[Theorem 2]{Ih1}. Hence the image of $\kappa_{(p-1,p-1)}(\sigma_{(p-1,p-1)})$ generates $\kappa_{(p-1,p-1)}(A)$, and it is nontrivial by Lemma \ref{lmm:nontrivial}. Finally, the case where $\boldsymbol{m}=(p,1)$ or $(1,p)$ follows from Lemma \ref{lmm:nontrivial} and Lemma \ref{lmm:(1,p)(p,1)}.
\end{proof}

We inductively define $\sigma_{\boldsymbol{m}}$ and $g_{\boldsymbol{m}}$ for general $\boldsymbol{m} \in I$ as follows: 

\medskip

{\bf Construction.}
\begin{itemize}
\item First, assume that $\boldsymbol{m}=(m_{1}, m_{2}) \in I$ is an index satisfying that $m_{1} \geq p$, $m_{2} \leq p-1$ and $\boldsymbol{m} \not \equiv \boldsymbol{1} \bmod p-1$. we define $\sigma_{\boldsymbol{m}}$ and $g_{\boldsymbol{m}}$ as 
\[
\sigma_{\boldsymbol{m}} \coloneqq \left( \gamma_{1}\sigma_{\boldsymbol{m}-(p-1,0)}\gamma_{1}^{-1}\sigma_{\boldsymbol{m}-(p-1,0)}^{-\chi_{1}^{m_{1}}(\gamma_{1})} \right)^{(1,m_{1})}
\] and 
\[
g_{\boldsymbol{m}} \coloneqq \gamma_{1}g_{\boldsymbol{m}-(p-1,0)}\gamma_{1}^{-1}g_{\boldsymbol{m}-(p-1,0)}^{-\chi_{1}^{m_{1}}(\gamma_{1})}.
\]
\item Secondly, if $\boldsymbol{m} \in I$ is an index such that $m_{2} \geq p$ and $\boldsymbol{m} \not \equiv \boldsymbol{1} \bmod p-1$, we define $\sigma_{\boldsymbol{m}}$ and $g_{\boldsymbol{m}}$ as
 \[
\sigma_{\boldsymbol{m}} \coloneqq \left(\gamma_{2}\sigma_{\boldsymbol{m}-(0,p-1)}\gamma_{2}^{-1}\sigma_{\boldsymbol{m}-(0,p-1)}^{-\chi_{2}^{m_{2}}(\gamma_{2})}\right)^{(2,m_{2})}
\] and 
\[
g_{\boldsymbol{m}} \coloneqq \gamma_{2}g_{\boldsymbol{m}-(0,p-1)}\gamma_{2}^{-1}g_{\boldsymbol{m}-(0,p-1)}^{-\chi_{2}^{m_{2}}(\gamma_{2})}.
\]
\end{itemize}

Finally, we consider the case where $\boldsymbol{m} \equiv \boldsymbol{1} \bmod p-1$. In this case, if $\boldsymbol{m} \geq (2,2)$, we obtain two candidates for $\sigma_{\boldsymbol{m}}$ obtained from applying the construction inductively from $\sigma_{(1,p)}$ and $\sigma_{(p,1)}$.

However, these candidates define the same element on the level of the abelianization $A$ of the Galois group of $\Omega$ over $K(p^{\infty})$. In fact, write $\boldsymbol{m}=(1+n_{1}(p-1), 1+n_{2}(p-1))$ for some $\boldsymbol{n}=(n_{1}, n_{2}) \in \mathbb{Z}_{\geq 1}^{2}$. If we start from $\sigma_{(p,1)}$ to obtain $\sigma_{\boldsymbol{m}}$ through the above construction, we have
\[ 
\sigma_{\boldsymbol{m}}=\prod_{i=1}^{n_{1}-1}(\gamma_{1}-\chi_{1}(\gamma_{1})^{1+i(p-1)})\prod_{j=0}^{n_{2}-1}(\gamma_{2}-\chi_{2}(\gamma_{2})^{1+j(p-1)})\sigma_{(p,1)}.
\] as an element of $A$. On the other hand, if we start from $\sigma_{(p,1)}$ we have
\[ 
\sigma_{\boldsymbol{m}}=\prod_{i=0}^{n_{1}-1}(\gamma_{1}-\chi_{1}(\gamma_{1})^{1+i(p-1)})\prod_{j=1}^{n_{2}-1}(\gamma_{2}-\chi_{2}(\gamma_{2})^{1+j(p-1)})\sigma_{(1,p)}.
\] Hence two candidates yield the same element on $A$ for every $\boldsymbol{n} \geq \boldsymbol{1}$ by Lemma \ref{Wintenberger}(2), and a similar argument can also be applied to the case of $g_{\boldsymbol{m}}$. Now we define $\sigma_{\boldsymbol{m}}$ and $g_{\boldsymbol{m}}$ for $\boldsymbol{m} \equiv \boldsymbol{1} \bmod p-1$ as follows.

\begin{itemize}
\item For an index $\boldsymbol{m}=(m,1)$ such that $m \geq 2$ and $m \equiv 1 \bmod p-1$, we define $\sigma_{(m,1)}$ and $g_{(m,1)}$ by applying the above construction, starting from $\sigma_{(p,1)}$. Similarly, for every index $\boldsymbol{m}=(1,m)$ such that $m \geq 2$ and $m \equiv 1 \bmod p-1$, we define $\sigma_{(1,m)}$ and $g_{(1,m)}$ by applying the above construction, starting from $\sigma_{(1,p)}$. 

\item Let $\boldsymbol{m} \in I$ be an index satisfying $\boldsymbol{m} \geq (2,2)$ and $\boldsymbol{m} \equiv \boldsymbol{1} \bmod p-1$. Then we define $\sigma_{\boldsymbol{m}}$ and $g_{\boldsymbol{m}}$ by applying the above construction, starting from $\sigma_{(p,1)}$. 
\end{itemize}

We have the following lemma.

\begin{lemma}\label{lmm:sigma=g}
For $\boldsymbol{m} \in I$, two elements $\sigma_{\boldsymbol{m}}$ and $g_{\boldsymbol{m}}$ define the same element of $A$.
\end{lemma}

\begin{proof}
Since $A$ is abelian, the elements $\epsilon_{1,m}$ and $\epsilon_{2,m}$ act on $A$ as idempotents for every $m$. Moreover, the action of $\epsilon_{i,m}$ commutes with the conjugation by $\gamma_{1}$ and $\gamma_{2}$. Hence the assertion follows from the construction of $\sigma_{\boldsymbol{m}}$ and $g_{\boldsymbol{m}}$.
\end{proof}

We now show that $\{ \sigma_{\boldsymbol{m}} \}_{\boldsymbol{m} \in I}$ satisfy the assumption of our conjecture \ref{cnj:DI11}:

\begin{proposition}\label{prp:lift}
For $\boldsymbol{m}=(m_{1}, m_{2}) \in I$, the following assertions hold.
\begin{enumerate}
\item[(1)] The element $\sigma_{\boldsymbol{m}}$ is contained in $F^{\boldsymbol{m}}G$, and the image of $\sigma_{\boldsymbol{m}}$ in the $|\boldsymbol{m}|$-th graded quotient $\mathfrak{g}_{\lvert \boldsymbol{m} \rvert}$ is contained in the $\chi^{\boldsymbol{m}}$-isotypic component of $\mathfrak{g}_{\lvert \boldsymbol{m} \rvert}$.
\item[(2)] The element $\kappa_{\boldsymbol{m}}(\sigma_{\boldsymbol{m}})$ generates $\kappa_{\boldsymbol{m}}(F^{|\boldsymbol{m}|}G)$, which is nonzero.
\end{enumerate}
\end{proposition}

\begin{proof}
First, by Lemma \ref{lmm:lift_1}, the assertion of (1) holds for every $\boldsymbol{m} \in I_{0}$. For every $\boldsymbol{m} \in I$ satisfying the assertion of (1), we have $\gamma_{1}\sigma_{\boldsymbol{m}}\gamma_{1}\sigma_{\boldsymbol{m}}^{-\chi_{1}^{m_{1}}(\gamma_{1})} \in F^{\boldsymbol{m}+(1,0)}G$, by Corollary \ref{cor:Gal_2}. Now we claim that 
\[
\sigma_{\boldsymbol{m}+(p-1,0)}=\left(\gamma_{1}\sigma_{\boldsymbol{m}}\gamma_{1}\sigma_{\boldsymbol{m}}^{-\chi_{1}^{m_{1}}(\gamma_{1})}\right)^{(1,m_{1})} \in F^{\boldsymbol{m}+(p-1,0)}G.
\] The element $\sigma_{\boldsymbol{m}+(p-1,0)}$ is contained in $F^{\boldsymbol{m}+(1,0)}G$. By the construction of $\sigma_{\boldsymbol{m}+(p-1,0)}$ (cf. Lemma \ref{lmm:idemp}), we have the equality
\[
\delta_{1} \sigma_{\boldsymbol{m}+(p-1,0)} \delta_{1}^{-1} = \chi_{1}^{m_{1}}(\delta_{1})\sigma_{\boldsymbol{m}+(p-1,0)}
\] in the two-variable graded quotient $\Gr^{\boldsymbol{m}+(1,0)}_{1}G$. However, by Corollary \ref{cor:Gal_2}, we also have
\[
\delta_{1} \sigma_{\boldsymbol{m}+(p-1,0)} \delta_{1}^{-1} = \chi_{1}^{m_{1}+1}(\delta_{1})\sigma_{\boldsymbol{m}+(p-1,0)}
\] in the same graded quotient. Since the element $\chi_{1}(\delta_{1}) \in \mathbb{Z}_{p}^{\times}$ is of order $p-1$, we have $\sigma_{\boldsymbol{m}+(p-1,0)} \in F^{\boldsymbol{m}+(2,0)}G$. By applying the same argument to graded quotients \[
\Gr^{\boldsymbol{m}+(2,0)}_{1}G, \dots, \Gr^{\boldsymbol{m}+(p-2,0)}_{1}G,
\] we obtain $\sigma_{\boldsymbol{m}+(p-1,0)} \in F^{\boldsymbol{m}+(p-1,0)}G$, as desired. A similar argument immediately shows that $\sigma_{\boldsymbol{m}} \in F^{\boldsymbol{m}}G$ for every $\boldsymbol{m} \in I$. This proves the former assertion of (1), and the latter assertion follows from Corollary \ref{cor:Gal_2}.

To prove the assertion of (2), it suffices to show that the element $\kappa_{\boldsymbol{m}}(g_{\boldsymbol{m}})$ generates $\kappa_{\boldsymbol{m}}(F^{|\boldsymbol{m}|}G)$ for every $\boldsymbol{m} \in I$ by Lemma \ref{lmm:sigma=g}. We fix an arbitrary index $\boldsymbol{m}_{0} \in I_{0} \setminus \{(1,p)\}$, and prove the assertion of (2) for every $\boldsymbol{m}$ such that $\boldsymbol{m} \equiv \boldsymbol{m}_{0} \bmod p-1$. By Lemma \ref{Wintenberger}, the $\Lambda$-module $A^{\boldsymbol{m}_{0}}$ is generated by (the image of) $g_{\boldsymbol{m}_{0}}$ if $\boldsymbol{m}_{0} \neq (p,1)$, and generated by $g_{(p,1)}$ and $g_{(1,p)}$ if $\boldsymbol{m}_{0}=(p,1)$.

For every $m \geq \lvert \boldsymbol{m}_{0} \rvert$ such that $m \equiv \lvert \boldsymbol{m}_{0} \rvert \bmod p-1$, let $A^{\boldsymbol{m}_{0}, m}$ be the image of $F^{m}G$ in $A^{\boldsymbol{m}_{0}}$. Note the image of $F^{m+1}G$ in $A^{\boldsymbol{m}_{0},m}$ coincides with $A^{\boldsymbol{m}_{0}, m+(p-1)}$. In fact, assume that there is an integer $1 \leq i < p-1$ such that the image of $F^{m+i}G$ properly contains $A^{\boldsymbol{m}_{0}, m+(p-1)}$, and take $i$ to be maximal among such integers. Then we would have a nontrivial homomorphism
\[
\mathfrak{g}_{m+i} \to A^{\boldsymbol{m}_{0}, m}/ A^{\boldsymbol{m}_{0}, m+p-1}.
\] However, since $\mathfrak{g}_{m+i}$ is embedded in $\mathfrak{g}_{m+i} \otimes \mathbb{Q}_{p}$, which is a direct sum of $\mathbb{Q}_{p}(n_{1}, n_{2})$ where $n_{1}+n_{2}=m+i \not \equiv m \bmod p-1$, the subspace $\epsilon_{1,m_{1}}\epsilon_{2,m_{2}}\mathfrak{g}_{m+i}$ is trivial. Hence the image of the above homomorphism should be trivial, which is a contradiction.

By what we have just proved, we have a homomorphism 
\[
\mathfrak{g}_{|\boldsymbol{m}|} \to A^{\boldsymbol{m}_{0}, |\boldsymbol{m}|}/ A^{\boldsymbol{m}_{0}, |\boldsymbol{m}|+p-1}
\] sending $\sigma_{\boldsymbol{m}}$ to $g_{\boldsymbol{m}}$, and the character $\kappa_{\boldsymbol{m}} \mid_{F^{\lvert \boldsymbol{m} \rvert}G}$ factors through $A^{\boldsymbol{m}_{0},\lvert \boldsymbol{m} \rvert}/A^{\boldsymbol{m}_{0},\lvert \boldsymbol{m} \rvert+(p-1)}$, for every $\boldsymbol{m} \in I$ such that $\boldsymbol{m} \equiv \boldsymbol{m}_{0} \bmod p-1$. Now we prove the following claim:

\medskip

{\bf Claim.} The $\Lambda$-module $A^{\boldsymbol{m}_{0},m}$ is generated by $\{ g_{\boldsymbol{m}} \}_{\boldsymbol{m}}$, where $\boldsymbol{m}$ ranges over the indexes $\boldsymbol{m} \in I$ such that $|\boldsymbol{m}|=m$ and $\boldsymbol{m} \equiv\boldsymbol{m}_{0} \bmod p-1$. Moreover, the element $\kappa_{\boldsymbol{m}}(g_{\boldsymbol{m}})$ generates the image of $\kappa_{\boldsymbol{m}}(A^{\boldsymbol{m}_{0},m}) \neq 0$. 

\medskip

Before giving the proof of the claim, we remark that, for every $\boldsymbol{m} \in I$ with $\boldsymbol{m} \equiv \boldsymbol{m}_{0} \bmod p-1$, the image of $g_{\boldsymbol{m}}$ in $A^{\boldsymbol{m}_{0},\lvert \boldsymbol{m} \rvert}/A^{\boldsymbol{m}_{0},\lvert \boldsymbol{m} \rvert+(p-1)}$ is contained in the $\chi^{\boldsymbol{m}}$-isotypic component, since it is the image of $\sigma_{\boldsymbol{m}} \in \mathfrak{g}_{|\boldsymbol{m}|}$, which is contained in the $\chi^{\boldsymbol{m}}$-isotypic component by (1). In particular, the second half of the claim follows from the first.

In the following, we prove the first half of the claim by induction on $m$. We identify the ring $\Lambda=\mathbb{Z}_{p}[[\Gal(K(p^{\infty})/K(p))]]$ with the power series ring $\mathbb{Z}_{p}[[S_{1}, S_{2}]]$ via $S_{i}=\gamma_{i}-1$ for $i=1,2$.

First, we know that the claim holds for $m=|\boldsymbol{m}_{0}|$. We prove the claim for $m+(p-1)$, assuming it is true for $m$. Take an arbitrary element $x \in A^{\boldsymbol{m}_{0}, m+(p-1)}$. By induction hypothesis, the element $x$ can be written as
\[
x= \sum f_{\boldsymbol{m}}(S_{1}, S_{2})g_{\boldsymbol{m}}
\] for some $f_{\boldsymbol{m}}(S_{1}, S_{2}) \in \mathbb{Z}_{p}[[S_{1}, S_{2}]]$ where $\boldsymbol{m}$ ranges over the indexes $\boldsymbol{m} \in I$ satisfying $|\boldsymbol{m}|=m$ and $\boldsymbol{m} \equiv \boldsymbol{m}_{0} \bmod p-1$. 

Since $x \in A^{\boldsymbol{m}_{0}, m+(p-1)}$, we have $\kappa_{\boldsymbol{n}}(x)=0$ for every $\boldsymbol{n}=(n_{1}, n_{2})$ such that $|\boldsymbol{n}|=m$ and $\boldsymbol{n} \equiv \boldsymbol{m}_{0} \bmod p-1$, by Lemma \ref{lmm:restrict} (2). Hence we have
\begin{align*}
\kappa_{\boldsymbol{n}}(x)
&=  \sum \kappa_{\boldsymbol{n}}(f_{\boldsymbol{m}}(S_{1}, S_{2})g_{\boldsymbol{m}}) \\
&= \sum f_{\boldsymbol{m}}(\chi_{1}^{n_{1}}(\gamma_{1})-1,\chi_{2}^{n_{2}}(\gamma_{2})-1 )\kappa_{\boldsymbol{n}}(g_{\boldsymbol{m}}) \\
&=  f_{\boldsymbol{n}}(\chi_{1}^{n_{1}}(\gamma_{1})-1,\chi_{2}^{n_{2}}(\gamma_{2})-1 )\kappa_{\boldsymbol{n}}(g_{\boldsymbol{n}})=0.
\end{align*}

Note that $\kappa_{\boldsymbol{n}}(g_{\boldsymbol{n}})$ generates $\kappa_{\boldsymbol{n}}(F^{\lvert \boldsymbol{n} \rvert}G)$ by induction hypothesis. Moreover, the submodule $\kappa_{\boldsymbol{n}}(F^{\lvert \boldsymbol{n} \rvert}G) \subset \mathbb{Z}_{p}(\boldsymbol{n})$ is nonzero by Lemma \ref{lmm:nontrivial}. Therefore we have $\kappa_{\boldsymbol{n}}(g_{\boldsymbol{n}}) \neq 0$, implying that $f_{\boldsymbol{n}}(\chi_{1}^{n_{1}}(\gamma_{1})-1,\chi_{2}^{n_{2}}(\gamma_{2})-1)=0$. Equivalently, the power series $f_{\boldsymbol{n}}(S_{1}, S_{2})$ is contained in the ideal $(S_{1}-\chi_{1}^{n_{1}}(\gamma_{1})+1, S_{2}-\chi_{2}^{n_{2}}(\gamma_{2})+1)$. Since we have
\begin{align*}
(S_{1}-\chi_{1}^{n_{1}}(\gamma_{1})+1)g_{\boldsymbol{n}}
&= (\gamma_{1}-\chi_{1}^{n_{1}}(\gamma_{1}))g_{\boldsymbol{n}} = g_{\boldsymbol{n}+(p-1,0)} \quad \text{and} \\
(S_{2}-\chi_{2}^{n_{2}}(\gamma_{2})+1)g_{\boldsymbol{n}}
&= (\gamma_{2}-\chi_{2}^{n_{2}}(\gamma_{2}))g_{\boldsymbol{n}} = g_{\boldsymbol{n}+(0,p-1)},
\end{align*} the element $f_{\boldsymbol{n}}(S_{1}, S_{2})g_{\boldsymbol{n}}$ can be written as a $\mathbb{Z}_{p}[[S_{1}, S_{2}]]$-linear combination of $g_{\boldsymbol{n}+(p-1,0)}$ and $g_{\boldsymbol{n}+(0,p-1)}$. By repeating this argument, it follows that $x$ can be expressed as a $\mathbb{Z}_{p}[[S_{1}, S_{2}]]$-linear combination of $\{ g_{\boldsymbol{m}} \}_{\boldsymbol{m}}$, where $\boldsymbol{m} \in I$ ranges over  the indexes  satisfying $|\boldsymbol{m}|=m+(p-1)$ and $\boldsymbol{m} \equiv \boldsymbol{m}_{0} \bmod p-1$. This concludes the proof of the claim.
\end{proof}

\begin{remark}
In our construction of $\{ \sigma_{\boldsymbol{m}} \}_{\boldsymbol{m} \in I}$, we assume the two assumptions
\begin{enumerate}
    \item  The class number of $K(p)$ is not divisible by $p$, and
    \item  There are exactly two primes of $K(p^{2})$ above $p$
\end{enumerate}
to apply Lemma \ref{Wintenberger}. However, it is clear from our construction that one can remove these two assumptions, if there exists a quotient $\bar{A}$ of the $\Lambda$-module $A$ that satisfies the following properties:
\begin{enumerate}
    \item[(a)] There exists a surjective homomorphism $\Lambda^{[K(p):K]-1} \oplus \Ann(\mathbb{Z}_{p}(1)) \to \bar{A}$, and
    \item[(b)] The character $\kappa_{\boldsymbol{m}}$ factors through $\bar{A}$.
\end{enumerate} 
For example, let $\bar{A}$ be the Galois group of a Kummer extension of $K(p^{\infty})$ obtained by adjoining all $p$-powerth roots of elliptic units in $K(p^{\infty})$. Then it satisfies (b), but we do not know whether it also satisfies (a) or not.
\end{remark}

\subsection{Group-theoretic lemmas and end of proof of main result}\label{4.2}

In this subsection, we complete a proof of Theorem \ref{thm:main}. First, we prepare a series of group-theoretic lemmas. We first prove a generalization of Lemma 3.1 in \cite{Sh} to the case of a free pro-$p$ group of countably infinite rank:

\begin{lemma}\label{lmm:group}
Let $\mathcal{F}$ be a pro-$p$ group strongly generated by $y$ and $\{ x_{i} \}_{i \geq 1}$. For each $i \geq 1$, let $x_{i,1} \coloneqq x_{i}$ and we inductively define 
\[
x_{i,j+1} \coloneqq [y, x_{i,j}]x_{i,j}^{pa_{i,j}}
\] for some $a_{i,j} \in \mathbb{Z}_{p}$ for every $j \geq 1$. Denote the normal closure of $\{ x_{i} \}_{i \geq 1}$ in $\mathcal{F}$ by $H$. The following assertions hold:
\begin{enumerate}
\item[(1)] $H$ is strongly generated by $\{ x_{i,j} \}_{i, j \geq 1}$.
\item[(2)] If $\mathcal{F}$ is a free pro-$p$ group on $y$ and $\{ x_{i} \}_{i \geq 1}$, then $H$ is a free pro-$p$ group on $\{ x_{i,j} \}_{i, j \geq 1}$.
\end{enumerate}
\end{lemma}

\begin{proof}
Let $\mathcal{K}$ be a free pro-$p$ group on the set $\{ \tilde{x}_{i,j} \}_{i, j \geq 1}$. We define a two-variable filtration on this group by
\[
\mathcal{K}_{i,j} \coloneqq \langle \tilde{x}_{i',j'} \mid i' \geq i \text{ or } j' \geq j \rangle_{\rm{normal}}
\] for every $i, j \geq 1$. The quotient $\mathcal{K}/\mathcal{K}_{i,j}$ is a free pro-$p$ group of finite rank, and the image of $\{ \tilde{x}_{i,j} \}_{i'<i,j'<j}$ forms a basis of this quotient.

We define an automorphism $\phi \colon \mathcal{K} \xrightarrow{\sim} \mathcal{K}$ by $\phi(\tilde{x}_{i,j}) \coloneqq \tilde{x}_{i,j+1}\tilde{x}_{i,j}^{1-pa_{i,j}}$ for every $i, j \geq 1$. Then it is straightforward to see that $\phi$ induces an automorphism $\mathcal{K}/\mathcal{K}_{i,j}$ for every $i,j$. Hence it defines an element of a profinite group $\Aut_{\rm{fil}}(\mathcal{K}) \coloneqq \varprojlim_{i,j} \Aut(\mathcal{K}/\mathcal{K}_{i,j})$, which is regarded as a subgroup of $\Aut(\mathcal{K})$. We extend a homomorphism $\mathbb{Z} \to \Aut_{\rm{fil}}(\mathcal{K})$ sending $1$ to the automorphism $\phi$ to a continuous homomorphism $\hat{\mathbb{Z}} \to \Aut_{\rm{fil}}(\mathcal{K})$. We claim that the resulting map factors through the maximal pro-$p$ quotient $\hat{\mathbb{Z}} \to \mathbb{Z}_{p}$.

Let us set $M \coloneqq \mathcal{K}^{\rm{ab}}/p$,  and let $M_{i,j} \subset M$ denote the image of $\mathcal{K}_{i,j}$ inside $M$. Since taking the maximal abelian quotients and taking reductions modulo $p$ are both right exact, the quotient $M/M_{i,j}$ is naturally isomorphic to $(\mathcal{K}/\mathcal{K}_{i,j})^{\rm{ab}}/p$. Note that the kernel of the homomorphism
\[
\Aut_{\rm{fil}}(\mathcal{K})=\varprojlim \Aut(\mathcal{K}/\mathcal{K}_{i,j})\to 
\Aut_{\rm{fil}}(M) \coloneqq \varprojlim \Aut(M/M_{i,j}) 
\] is pro-$p$, since the kernel of each $\Aut(\mathcal{K}/\mathcal{K}_{i,j}) \to \Aut(M/M_{i,j})$ is so. Therefore, to prove the claim, it suffices to show that the homomorphism $\hat{\mathbb{Z}} \to \Aut_{\rm{fil}}(M)$ corresponding to the image of $\phi$ factors through $\mathbb{Z}_{p}$.

By the construction of the automorphism $\phi$, it follows that \[
\phi^{p^{n}} \in \Ker \left(\Aut_{\rm{fil}}(M) \to \Aut(M/M_{p^{n}, p^{n}})\right)
\] for every $n \geq 1$. In fact, a direct computation shows that $\phi^{p}(\tilde{x}_{i,j})=\tilde{x}_{i,j+p}+\tilde{x}_{i,j}$ in $M$ for every $i,j$, and iterating $\phi^{p}$ verifies the claim. Hence the image of $\phi$ in $\Aut_{\rm{fil}}(M)=\Aut(M/M_{p^{n},p^{n}})$ is a pro-$p$ group. It follows that the map $\mathbb{Z} \to \Aut_{\rm{fil}}(\mathcal{K})$ corresponding to $\phi$ naturally extends to a continuous homomorphism $\mathbb{Z}_{p} \to \Aut_{\rm{fil}}(\mathcal{K})$, and we can take the associated semi-direct product $\mathcal{K} \rtimes \mathbb{Z}_{p}$. 

There is a unique homomorphism $\mathcal{K} \to H$ sending $\tilde{x}_{i,j}$ to $x_{i,j}$ for every $i, j \geq 1$ by the freeness of $H$. Since the action of $\phi$ on $\mathcal{K}$ is compatible with the conjugation by $y$ on $H$, we can extend this homomorphism to $\mathcal{K} \rtimes \mathbb{Z}_{p} \to \mathcal{F}$ by mapping $\phi$ to $y$, and the resulting homomorphism is surjective by construction. Hence $\mathcal{K} \to H$ is also surjective. This proves the assertion of (1). Moreover, if $\mathcal{F}$ is a free pro-$p$ group on $y$ and $\{ x_{i} \}_{i}$, the universal property of $\mathcal{F}$ leads to the existence of the inverse of $\mathcal{K} \rtimes \mathbb{Z}_{p} \to \mathcal{F}$ constructed above. Hence $H$ is isomorphic to $\mathcal{K}$, proving the second assertion.
\end{proof}

\begin{corollary}\label{cor:group}
Let $r$ be a positive integer and $\mathcal{F}$ a pro-$p$ group generated by $y_{1}, y_{2}$ and $\{ x_{i} \}_{1 \leq i \leq r}$. For each $1 \leq i \leq r$, let $x_{i, (0,0)} \coloneqq x_{i}$ and we inductively define 
\[
x_{i,(j+1,0)} \coloneqq [y_{1}, x_{i,(j,0)}]x_{i,(j,0)}^{pa_{i,j}}
\] for some $a_{i,j} \in \mathbb{Z}_{p}$ and $j \geq 0$. Similarly, for each $1 \leq i \leq r$ and each $j \geq 0$, define
\[
x_{i,(j,k+1)} \coloneqq [y_{2}, x_{i,(j,k)}]x_{i,(j,k)}^{pb_{i,j,k}}
\] for some $b_{i,j,k} \in \mathbb{Z}_{p}$ and $k \geq 0$. Moreover, let $z_{(0,0)} \coloneqq [y_{1}, y_{2}]$ and define
\[
z_{(i+1,0)} \coloneqq [y_{1}, z_{(i,0)}]z_{(i,0)}^{p\alpha_{i}}
\] for some $\alpha_{i} \in \mathbb{Z}_{p}$ and $i \geq 0$. Finally, for each $i \geq 0$, we define
\[
z_{(i,j+1)} \coloneqq [y_{2}, z_{(i,j)}]z_{(i,j)}^{p\beta_{i,j}}
\] for some $\beta_{i,j} \in \mathbb{Z}_{p}$ and each $j \geq 0$. Let $H$ denote the normal closure of $\{ x_{i} \}_{1 \leq i \leq r}$ and $z_{(0,0)}$ inside $\mathcal{F}$. Then $H$ is strongly generated by $\{ x_{i,(j,k)} \}_{\substack{1 \leq i \leq r, \\ j, k \geq 0}}$ and $\{ z_{i,j} \}_{i,j \geq 0}$.
\end{corollary}

\begin{proof}
We may assume that $\mathcal{F}$ is a free pro-$p$ group on $y_{1}, y_{2}$ and $\{ x_{i} \}_{1 \leq i \leq r}$. First, let $\mathcal{F}_{1}$ be the kernel of the following homomorphism:
\[
 \mathcal{F} \to \mathbb{Z}_{p};  \; y_{1} \mapsto 1, y_{2} \mapsto 0, x_{i} \mapsto 0 \quad (1 \leq i \leq r)
\] Then, by Lemma \ref{lmm:group}, the group $\mathcal{F}_{1}$ is a free pro-$p$ group on $y_{2}$, $\{ z_{(i,0)} \}_{i \geq 0}$ and $\{ x_{i,(j,0)} \}_{\substack{1\leq i \leq r,\\ j \geq 0}}$. Moreover, the group $H$ coincides the kernel of the homomorphism 
\[
\mathcal{F}_{1} \to \mathbb{Z}_{p}; \; y_{2} \mapsto 1, z_{(i,0)} \mapsto 0 \; (i \geq 0) \text{ and } x_{i,(j,0)} \mapsto 0 \; (1 \leq i \leq r, j \geq 0).
\] By applying Lemma \ref{lmm:group} again, it follows that $H$ is a free pro-$p$ group on $\{ z_{i,j} \}_{i \geq 0, j \geq 0}$ and $\{ x_{i,(j,k)}\}_{\substack{1 \leq i \leq r \\ j, k \geq 0}}$. This concludes the proof.
\end{proof}

The following lemma is used to compare the filtration $\{ F^{m}G \}_{m \geq 1}$ on $G$ with a certain canonical filtration on $G$ associated to $\{ \sigma_{\boldsymbol{m}} \}_{\boldsymbol{m} \in I}$. 

\begin{lemma}\label{lmm:filt}
Let $\mathcal{G}$ be a free pro-$p$ group on the set $\{ \tilde{\sigma}_{\boldsymbol{m}} \}_{\boldsymbol{m} \in I}$.
\begin{enumerate}
\item[(1)] There exists a unique descending central filtration $\{ \tilde{F}^{m}\mathcal{G} \}_{m \geq 1}$ on $\mathcal{G}$ satisfying the following property: $(\mathrm{i})$ $\tilde{\sigma}_{\boldsymbol{m}} \in \tilde{F}^{\lvert \boldsymbol{m} \rvert}\mathcal{G}$ for every $\boldsymbol{m} \in I$. $(\mathrm{ii})$ If $\{ F^{m}\mathcal{G} \}_{m \geq 1}$ is a descending central filtration satisfying $(\mathrm{i})$, then $\tilde{F}^{m}\mathcal{G} \subset F^{m}\mathcal{G}$ for every $m \geq 1$.
\item[(2)] The graded Lie algebra $\bigoplus_{m \geq 1}\tilde{F}^{m}\mathcal{G}/\tilde{F}^{m+1}\mathcal{G}$ is freely generated by the image of $\{ \tilde{\sigma}_{\boldsymbol{m}} \}_{\boldsymbol{m} \in I}$.
\item[(3)] The intersection $\cap_{m \geq 1} \tilde{F}^{m}\mathcal{G}$ is trivial.
\end{enumerate}
\end{lemma}
\begin{proof}
We construct $\{ \tilde{F}^{m}\mathcal{G} \}_{m \geq 1}$ as follows: First, let $\tilde{F}^{1}\mathcal{G} \coloneqq \mathcal{G}$. For $m \geq 2$, we inductively define $\tilde{F}^{m}\mathcal{G}$ as
\[
\tilde{F}^{m}\mathcal{G} \coloneqq \langle \{\tilde{\sigma}_{\boldsymbol{m}} \}_{\lvert \boldsymbol{m} \rvert \geq m}, \{ [\tilde{F}^{m'}\mathcal{G}, \tilde{F}^{m''}\mathcal{G}] \}_{\substack{m'<m, m''<m \\ m \leq m'+m''}} \rangle_{{\rm normal}}.
\] 
Since $[\tilde{F}^{m}\mathcal{G}, \tilde{F}^{1}\mathcal{G}] \subset \tilde{F}^{m}\mathcal{G}$, it follows that $\tilde{F}^{m+1}\mathcal{G} \subset \tilde{F}^{m}\mathcal{G}$ for every $m \geq 1$. Apparently, the filtration $\{ \tilde{F}^{m}\mathcal{G} \}_{m \geq 1}$ defines a descending central filtration on $\mathcal{G}$ that satisfies the condition ($\mathrm{i}$) in the assertion of (1). Now let $\{ F^{m}\mathcal{G} \}_{m \geq 1}$ be an arbitrary descending central filtration on $\mathcal{G}$ that satisfies ($\mathrm{i}$). We have $F^{1}\mathcal{G}=\tilde{F}^{1}\mathcal{G}=\mathcal{G}$. By induction on $m$, it immediately follows that $\tilde{F}^{m}\mathcal{G} \subset F^{m}\mathcal{G}$ holds for every $m \geq 1$. Hence the filtration $\{ \tilde{F}^{m}\mathcal{G} \}_{m \geq 1}$ also satisfies ($\mathrm{ii}$). The uniqueness is clear. The proof of the assertion of (2) is similar to that of \cite[p.263, 5]{Ih6}.

We prove the last assertion. Let $F_{n}$ be a free pro-$p$ group on the set $\{ \tilde{\sigma}_{\boldsymbol{m}} \}_{\boldsymbol{m} \in I, \lvert \boldsymbol{m} \rvert \leq n}$ for every $n \geq 2$. Since $\mathcal{G}$ is isomorphic to the projective limit $\varprojlim_{n} F_{n}$, it suffices to show that the image of the intersection $\cap_{m \geq 1} \tilde{F}^{m}G$ in $F_{n}$ is trivial for every $n \geq 2$.

For an integer $m \geq n$, the image of $\tilde{F}^{m}\mathcal{G}$ in $F_{n}$ is normally generated by the image of
\[
\{ [\tilde{F}^{m'}\mathcal{G}, \tilde{F}^{m''}\mathcal{G}] \}_{\substack{m'<m, m''<m \\ m \leq m'+m''}} .
\] We claim that the image of $\tilde{F}^{m}\mathcal{G}$ in $F_{n}$ contained in the $r_{m}$-th component $F_{n}(r_{m})$ of the descending central series of $F_{n}$, where  $r_{m}$ is defined by
\[
r_{m} \coloneqq  \Bigl\lfloor \frac{m}{n} \Bigr\rfloor+1
\] for every $m \geq 1$. Once the claim is obtained, the image of $\cap_{m \geq 1} \tilde{F}^{m}\mathcal{G}$ in $F_{n}$ is contained in the intersection $\cap_{m \geq n} F_{n}(r_{m})=\{ 1 \}$, and the assertion follows.

The claim trivially holds for every $m \leq 2n-1$. Assume that the claim also holds for every $m \leq kn-1$ for some $k \geq 2$. We prove the claim for $m=kn, kn+1, \dots (k+1)n-1$ in order. Write $m=kn+r$ for some $0 \leq r \leq n-1$.

Let $m', m''$ be positive integers which are less than $m$ and $m'+m'' \geq m$, and write $m'=k'n+r'$, $m'=k''n+r''$ for some $0 \leq k', k'' \leq k$ and $0 \leq r', r'' \leq n-1$. Since
\[
m'+m''=(k'+k'')n+(r'+r'') \geq m=kn+r,
\] It holds that
\[
r_{m'}+r_{m''}=(k'+k'')+2 \geq (k+1)+\frac{n+r-(r'+r'')}{n}.
\] Since $\frac{n+r-(r'+r'')}{n} \geq \frac{2-n}{n} > -1$, it holds that
\[
r_{m'}+r_{m''} \geq r_{m}.
\] Therefore, the image of $[\tilde{F}^{m'}\mathcal{G}, \tilde{F}^{m''}\mathcal{G}]$ is contained in 
\[ [F_{n}(r_{m'}), F_{n}(r_{m''})] \subset F_{n}(r_{m'}+r_{m''}) \subset F_{n}(r_{m}).\] Hence the image of $\tilde{F}^{m}\mathcal{G}$ is contained in $F_{n}(r_{m})$, as desired.
\end{proof}

The next lemma gives an explicit set of generators of  $G=\Gal(\Omega/K(p))$:

\begin{lemma}\label{lmm:generator}
The group $G$ is generated by $\gamma_{1}$, $\gamma_{2}$ and $\{ \sigma_{\boldsymbol{m}} \}_{\boldsymbol{m} \in I_{0} \setminus \{ (p-1,p-1) \}}$.
\end{lemma}
\begin{proof}
We use Lemma \ref{Wintenberger} and definitions of $\sigma_{\boldsymbol{m}}$. It follows that the image of $\sigma_{\boldsymbol{m}}$ in $A=\Gal(\Omega/K(p^{\infty}))^{{\rm ab}}$ generates the $\Lambda$-module $A^{\boldsymbol{m}}$ for every $\boldsymbol{m} \in I_{0} \setminus \{ (p,1), (1,p),(p-1,p-1)\}$. Moreover, the image of $\sigma_{(p,1)}$ and $\sigma_{(1,p)}$ also generates $A^{(1,1)}$. Finally, recall that $\sigma_{(p-1,p-1)}$ coincides with the commutator $[\gamma_{1}, \gamma_{2}]$ and it generates $A^{(p-1,p-1)}$.

Hence the $\Lambda$-module $A$ is generated by the image of $\{ \sigma_{\boldsymbol{m}} \}_{\boldsymbol{m} \in I_{0}}$, which is equivalent to saying that $\Gal(\Omega/K(p^{\infty}))$ is normally generated by $\{ \sigma_{\boldsymbol{m}} \}_{\boldsymbol{m} \in I_{0}}$ (as a subgroup of $G$). Hence $\gamma_{1}$, $\gamma_{2}$ and $\{ \sigma_{\boldsymbol{m}} \}_{\boldsymbol{m} \in I_{0} \setminus \{ (p-1,p-1) \}}$ generate the group $G$, as desired.
\end{proof}

Finally, we give a proof of Theorem \ref{thm:main}. 

\begin{proof}[Proof of Theorem \ref{thm:main}]
    By Lemma \ref{lmm:generator}, the group $G$ is generated by 
    \[\gamma_{1}, \gamma_{2} \quad \text{and} \quad \{ \sigma_{\boldsymbol{m}}(=g_{\boldsymbol{m}}) \}_{\boldsymbol{m} \in I_{0} \setminus \{(p-1,p-1) \}}.
    \]  Now we recall the discussion occurring at constructing $\sigma_{\boldsymbol{m}}$ and $g_{\boldsymbol{m}}$ for $\boldsymbol{m} \in I$ such that $\boldsymbol{m} \equiv \boldsymbol{1} \bmod p-1$. There were two ways to construct them, starting from $\sigma_{(p,1)}$ and $\sigma_
{(1,p)}$, but they yield the same elements on $A=\Gal(\Omega/K(p^{\infty}))^{{\rm ab}}$. Hence, by applying Corollary \ref{cor:group} to $\mathcal{F}=G$, $y_{1}=\gamma_{1}$, $y_{2}=\gamma_{2}$ and $\{ x_{i} \}_{1 \leq i \leq r}=\{ g_{\boldsymbol{m}} \}_{\boldsymbol{m} \in I_{0}}$, it follows that $F^{1}G=\Gal(\Omega/K(p^{\infty}))$ is strongly generated by $\{ g_{\boldsymbol{m} }\}_{\boldsymbol{m} \in I}$. By Lemma \ref{lmm:sigma=g}, the elements $\{ \sigma_{\boldsymbol{m}} \}_{\boldsymbol{m} \in I}$  strongly generates $F^{1}G$.

We claim that the surjective map $F^{1}G \to F^{1}G^{\ast}$ is an isomorphism, which is equivalent to the desired equality $\Omega=\Omega^{\ast}$. To prove the claim, let $\mathcal{G}$ be a free pro-$p$ group on the set $\{ \tilde{\sigma}_{\boldsymbol{m}} \}_{\boldsymbol{m} \in I}$,  $\varphi \colon \mathcal{G} \to F^{1}G$ a surjective homomorphism sending $ \tilde{\sigma}_{\boldsymbol{m}}$ to $\sigma_{\boldsymbol{m}}$ for every $\boldsymbol{m} \in I$ and $\varphi^{\ast} \colon \mathcal{G} \to F^{1}G^{\ast}$ the composite of $\varphi$ and $F^{1}G \to F^{1}G^{\ast}$. 

The group $\mathcal{G}$ comes equipped with a descending central filtration $\{F^{m}\mathcal{G} \}_{m \geq 1}$ induced by $\{ F^{m}G^{\ast} \}_{m \geq 1}$ through the map $\varphi^{\ast}$. Note that the graded Lie algebra associated to $\{F^{m}\mathcal{G} \}_{m \geq 1}$ is the same as $\mathfrak{g}$, and we have
\[
\Ker(\varphi^{\ast})=(\varphi^{\ast})^{-1}(\bigcap_{m \geq 1}F^{m}G^{\ast})= \bigcap_{m \geq 1} F^{m}\mathcal{G}. 
\] By Lemma \ref{lmm:filt} (1), the group $F^{m}\mathcal{G}$ contains $\tilde{F}^{m}\mathcal{G}$ for every $m \geq 1$.  Hence we have the following commutative diagram:
\[
\begin{tikzcd}
\left( \bigoplus_{m \geq 1} \tilde{F}^{m}\mathcal{G}/\tilde{F}^{m+1}\mathcal{G} \right) \otimes \mathbb{Q}_{p}  \arrow[r] & \left( \bigoplus_{m \geq 1} F^{m}\mathcal{G}/F^{m+1}\mathcal{G} \right) \otimes \mathbb{Q}_{p} \cong \mathfrak{g} \otimes \mathbb{Q}_{p}  \\
\bigoplus_{m \geq 1} \tilde{F}^{m}\mathcal{G}/\tilde{F}^{m+1}\mathcal{G} \arrow[r]\arrow[u, hook] & \bigoplus_{m \geq 1} F^{m}\mathcal{G}/F^{m+1}\mathcal{G} \arrow[u, hook].
\end{tikzcd}
\]
Since we assume the analogue of the Deligne-Ihara conjecture (Conjecture \ref{cnj:DI11}), it holds that $\{ \tilde{\sigma}_{\boldsymbol{m}} \}_{\boldsymbol{m} \in I}$ freely generates $\bigoplus_{m \geq 1} (F^{m}\mathcal{G}/F^{m+1}\mathcal{G})\otimes \mathbb{Q}_{p}$ as a graded Lie algebra. However, by Lemma \ref{lmm:filt} (2), the Lie algebra $\bigoplus_{m \geq 1} (\tilde{F}^{m}\mathcal{G}/\tilde{F}^{m+1}\mathcal{G})\otimes \mathbb{Q}_{p}$ is also generated by $\{ \tilde{\sigma}_{\boldsymbol{m}} \}_{\boldsymbol{m} \in I}$, which forces the upper horizontal arrow in the diagram to be an isomorphism. As a consequence, the lower horizontal arrow in the diagram is found to be injective. As is observed by induction on $m$, this injectivity is equivalent to saying that two filtrations $\{ \tilde{F}^{m}\mathcal{G} \}_{m \geq 1}$ and $\{ F^{m}\mathcal{G} \}_{m \geq 1}$ coincide with each other. Hence by Lemma \ref{lmm:filt} (3), we have
\[
\bigcap_{m \geq 1}F^{m}\mathcal{G}=\bigcap_{m \geq 1}\tilde{F}^{m}\mathcal{G}=\{ 1 \},
\] showing the injectivity of $\varphi^{\ast}$. Hence the map $\varphi^{\ast}$ is an isomorphism, so is the map $F^{1}G \to F^{1}G^{\ast}$ as desired.
\end{proof} 

Lastly, we record the following corollary:
\begin{corollary}\label{cor:integral}
    The following assertions hold.
    \begin{enumerate}
        \item The elements $\{\sigma_{\boldsymbol{m}} \}_{\boldsymbol{m} \in I}$ freely generate the integral Lie algebra $\mathfrak{g}$. In other words, under the assumption of Theorem \ref{thm:main}, the integral version of Conjecture \ref{cnj:DI11} holds.
        \item The Galois group $\Gal(\Omega/K(p^{\infty}))$ is free pro-$p$ group with $\{ \sigma_{\boldsymbol{m}} \}_{\boldsymbol{m} \in I}$ as its basis (note that this result is consistent with Proposition \ref{prp:Iwasawa} (2)). 
    \end{enumerate}
\end{corollary}
\begin{proof}
    Both assertions follow from the proof of Theorem \ref{thm:main} above, together with Lemma \ref{lmm:filt}.
\end{proof}

\begin{remark}
     The integral version of the Deligne-Ihara conjecture is related to generalized Greenberg's conjecture for $\mathbb{Q}(\mu_{p})$ \cite[Theorem 1.3]{Sh}. Roughly speaking, if the integral version of the Deligne-Ihara conjecture holds, then the image of the pro-$p$ outer Galois representation from $G_{\mathbb{Q}(\mu_{p})}$ is found to be a free pro-$p$ group of rank $\frac{p+1}{2}$. However, such a free pro-$p$ extension cannot exist if $p>2$ is irregular and generalized Greenberg's conjecture \cite[Conjecture (3.5)]{Gr01} holds, by a result of McCallum \cite[Theorem 2]{Mc}. Similarly, one can show that the integral version of Conjecture \ref{cnj:DI11} also does not hold under similar assumptions. We plan to discuss this topic in a future paper.
\end{remark}

\appendix
\section{Pure locality of $p$-th cyclotomic extension of imaginary quadratic field}\label{appB}

We keep the same notation as in the main text: Let $K$ be an imaginary quadratic field of class number one, $p$ a prime $\geq 5$ which splits in $K$ as $(p)=\p \bar{\p}$, $\mu_{p}$ the group of $p$-th roots of unity and $K_{\infty}$ the $\mathbb{Z}_{p}^{2}$-extension of $K$. In this appendix, we determine the structure of the Galois group of the maximal pro-$p$ extension $\Omega^{\mathrm{cyc}}_{K}$ of $K(\mu_{p})$ unramified outside $p$ under a certain assumption.

We denote a unique prime of $K(\mu_{p})$ lying above $\p$ (resp. $\bar{\p}$) by the same letter, and fix an arbitrary prime of $\Omega_{K}^{\mathrm{cyc}}$ lying above $\p$. We obtain a homomorphism 
\[
\phi_{\p}: G_{\mathbb{Q}_{p}(\mu_{p})}^{(p)} \to \Gal(\Omega_{K}^{\mathrm{cyc}}/K(\mu_{p})),
\] associated to the decomposition group of the fixed prime of $\Omega_{K}^{\mathrm{cyc}}$. The Galois group is said to be \emph{purely local} if $\phi_{\p}$ is an isomorphism \cite{Wi} (see also \cite[(10.9.6) Definition]{NSW}). The aim of the present appendix is to prove the following, which tells us exactly when the Galois group is purely local:

\begin{theorem}\label{thm:app}
    The Galois group $\Gal(\Omega_{K}^{\mathrm{cyc}}/K(\mu_{p}))$ is purely local if and only if $p$ does not divide the class number of $K(\mu_{p})$ and there exists a unique prime of $K_{\infty}$ above $\p$. 
\end{theorem}

\begin{remark}
    (1) The Galois group $G_{\mathbb{Q}(\mu_{p})}^{(p)}$ is a pro-$p$ Demushkin group of rank $p+1$ \cite[(7.5.11) Theorem (ii)]{NSW}. In particular, it is generated by $p+1$ elements that satisfy one relation. On the other hand, if $p$ does not divide the class number of $K(\mu_{p})$, then $\Gal(\Omega_{K}^{\mathrm{cyc}}/K(\mu_{p}))$ is also generated by $p+1$ elements that satisfy one relation \cite[(10.7.13) Theorem]{NSW}. Hence the first condition on the class number implies that both Galois groups have the same numbers of generators and relations.

    (2) The latter condition on the number of primes of $K_{\infty}$ above $\p$ fails when $K=\mathbb{Q}(\sqrt{-1})$ and $p=29789$ \cite[Example 4.16]{IS1/2}. This is the only example the author has been able to find where the latter condition does not hold for $K=\mathbb{Q}(\sqrt{-1})$. 
\end{remark}

\begin{proof}[Proof of Theorem \ref{thm:app}]
    Assume that $\phi_{\p}$ is an isomorphism. Then there exists a unique prime of $\Omega^{\mathrm{cyc}}_{K}$ lying above $\p$. Since $K_{\infty}$ is a subfield of $\Omega^{\mathrm{cyc}}_{K}$, a prime of $K_{\infty}$ lying above $\p$ is uniquely determined. Moreover, the class number of $K(\mu_{p})$ is not divisible by $p$ by \cite[Proposition 4.13]{IS1/2}.

    Conversely, assume that $p$ does not divide the class number of $K(\mu_{p})$ and there exists a unique prime of $K_{\infty}$ above $\p$. By \cite[Proposition 4.13]{IS1/2}, it suffices to show that $\phi_{\p}$ is surjective. As a first step, we consider the kernel of the restriction map
    \[
       H^{1}(\Gal(\Omega^{\mathrm{cyc}}_{K}/K(\mu_{p})), \mu_{p}) \to H^{1}(G_{\mathbb{Q}_{p}(\mu_{p})}^{(p)}, \mu_{p}) \xrightarrow{\sim} \mathbb{Q}_{p}(\mu_{p})^{\times}/p,
    \] induced by $\phi_{\p}$, where the latter isomorphism comes from Kummer theory. Since the class number of $K(\mu_{p})$ is prime to $p$, Kummer theory gives an isomorphism
    \[
    O_{K(\mu_{p})}[1/p]^{\times}/p \xrightarrow{\sim}
    H^{1}(\Gal(\Omega^{\mathrm{cyc}}_{K}/K(\mu_{p})), \mu_{p}).
    \] Hence the surjectivity of $\phi_{\p}$ is equivalent to the injectivity of the map
    \[
    \iota_{\p}:
    O_{K(\mu_{p})}[1/p]^{\times}/p \to  \mathbb{Q}_{p}(\mu_{p})^{\times}/p 
    \] induced by the $\p$-adic completion. The map $\iota_{\p}$ is compatible with the action ofthe decomposition group of $\Gal(K(\mu_{p})/K)$ at $\p$, which is nothing but $\Gal(K(\mu_{p})/K)$ itself. Let
    \[
    \iota_{\bar{\p}}:
    O_{K(\mu_{p})}[1/p]^{\times}/p \to  \mathbb{Q}_{p}(\mu_{p})^{\times}/p
    \] be the map induced by the $\bar{\p}$-adic completion. Then the map
    \[
    (\iota_{\p}, \iota_{\bar{\p}}):
        O_{K(\mu_{p})}[1/p]^{\times}/p \to  \mathbb{Q}_{p}(\mu_{p})^{\times}/p \times \mathbb{Q}_{p}(\mu_{p})^{\times}/p. 
    \] is injective, since the dimension of the kernel of $(\iota_{\p}, \iota_{\bar{\p}})$ coincides with that of the $p$-class group of $K(\mu_{p})$ modulo $p$ (see \cite[Proof of (10.7.3)]{NSW}, for example). On the other hand, this injectivity already implies that the kernel of $\iota_{\p}$ is at most a single Tate twist $\mathbb{F}_{p}(1)$: More precisely, we claim that the composite
    \[
    \iota_{\p^{+}}:
     O_{K(\mu_{p})^{+}}[1/p]^{\times}/p \to O_{K(\mu_{p})}[1/p]^{\times}/p \xrightarrow{\iota_{\p}}  \mathbb{Q}_{p}(\mu_{p})^{\times}/p.
    \] is injective, where $K(\mu_{p})^{+}$ denotes the maximal real subfield of $K(\mu_{p})$. Note that, by Dirichlet's unit theorem, we have
    \[
    \dim_{\mathbb{F}_{p}}O_{K(\mu_{p})^{+}}[1/p]^{\times}/p=p
    \quad \text{and} \quad
    \dim_{\mathbb{F}_{p}} O_{K(\mu_{p})}[1/p]^{\times}/p=p+1,
    \] and the cokernel of the inclusion 
    $ O_{K(\mu_{p})^{+}}[1/p]^{\times}/p \to O_{K(\mu_{p})}[1/p]^{\times}/p
    $ is generated by $\mu_{p} \cong \mathbb{F}_{p}(1)$. The injectivity of $\iota_{\p^{+}}$ follows from the that of the map $(\iota_{\p}, \iota_{\bar{\p}})$ and the fact that a prime of $K(\mu_{p})^{+}$ lying above $p$ is unique, since the restriction of $(\iota_{\p}, \iota_{\bar{\p}})$ on $O_{K(\mu_{p})^{+}}[1/p]^{\times}/p$ is regarded as the diagonal map induced by the completion at the unique prime $K(\mu_{p})^{+}$ lying above $p$.

    Now it is easy to check by using \cite[(8.7.2) Proposition]{NSW} that
    \[
    O_{K(\mu_{p})}[1/p]^{\times}/p \cong \mathbb{F}_{p}^{ \oplus 2} \oplus \mathbb{F}_{p}(1)^{\oplus 2} \oplus \bigoplus_{m=2}^{p-2} \mathbb{F}_{p}(m)
    \] as an $\mathbb{F}_{p}[\Gal(K(\mu_{p})/K)]$-module. Since $\Ker(\iota_{\p})$ is stable under the action of $\Gal(K(\mu_{p})/K)$, the injectivity of $\iota_{\p^{+}}$ implies that $\Ker(\iota_{\p})$ is contained in $\mathbb{F}_{p}(1)^{\oplus 2}$. In terms of Galois groups, we have proved that the cokernel of 
    \[
    G_{\mathbb{Q}_{p}(\mu_{p})}^{(p), \ab}/p \to \Gal(\Omega_{K}^{\mathrm{cyc}}/K(\mu_{p}))^{\ab}/p
    \] induced by $\phi_{\p}$ is contained in $\mathbb{F}_{p}^{2}$ (with trivial $\Gal(K(\mu_{p})/K)$-action). However, since this $\mathbb{F}_{p}^{2}$-extension is given by the modulo-$p$ quotient
    \[ 
    \Gal(K_{\infty}K(\mu_{p})/K(\mu_{p})) \cong \mathbb{Z}_{p}^{2} \to \mathbb{F}_{p}^{2}
    \] and we have assumed that there is a unique prime of $K_{\infty}K(\mu_{p})$ lying above $\p$, the map $\phi_{\p}$ is surjective.
\end{proof}

Theorem \ref{thm:app} allows us to use the local Tate duality to compute cohomology groups of $\Gal(\Omega^{\mathrm{cyc}}_{K}/K(\mu_{p}))$. Consequently, we obtain the following corollary:

\begin{corollary}\label{cor:app1}
    We keep the same assumption as in Theorem \ref{thm:app}. Then the cohomology group $H^{2}_{\et}(O_{K}[1/p], \mathbb{Z}_{p}(m_{1},m_{2}))$ is finite for every $(m_{1}, m_{2}) \in I$ such that $m_{1} \equiv m_{2} \bmod p-1$. Consequently, the $(m_{1},m_{2})$-th elliptic Soul\'e character $\kappa_{(m_{1}, m_{2})}$ is nontrivial for every such $(m_{1}, m_{2}) \in I$.
\end{corollary}

\begin{proof}
    The assertion follows from Theorem \ref{thm:app} and Theorem \ref{thm:1}.
\end{proof}

\bibliographystyle{amsalpha}
\bibliography{References}                                        

\end{document}